\documentclass[12pt]{article}
\usepackage{amsmath,amssymb,amsfonts,amsthm}
\usepackage{mathrsfs}

\usepackage{lmodern}
\usepackage{dsfont}

\usepackage{mathtools, nccmath, relsize}

\usepackage{indentfirst}
\usepackage[pdftex]{color,graphicx}

\usepackage{xcolor}
\definecolor{lightblue}{rgb}{0,0.2,0.5}
\usepackage[colorlinks=true, urlcolor=lightblue,linkcolor=lightblue, citecolor=lightblue]{hyperref}
\usepackage{ucs}

\definecolor{ForestGreen}{RGB}{34,139,34}
\definecolor{mauve}{rgb}{0.7,0,0.43}
\definecolor{dkgreen}{rgb}{0,0.6,0}
\definecolor{darkgreen}{rgb}{0,0.6,0}
\definecolor{darkorange}{rgb}{1.0, 0.55, 0.0}
\definecolor{lightblue}{rgb}{0,0.2,0.5}
\definecolor{blue1}{rgb}{0,0.1,0.9}
\usepackage{xcolor}

\definecolor{lightblue}{rgb}{0,0.2,0.5}
\usepackage[colorlinks=true, urlcolor=lightblue,linkcolor=lightblue, citecolor=lightblue]{hyperref}
\usepackage{ucs}

\definecolor{dblackcolor}{rgb}{0.0,0.0,0.0}
\definecolor{dbluecolor}{rgb}{0.01,0.02,0.7}
\definecolor{dgreencolor}{rgb}{0.2,0.4,0.0}
\definecolor{dgraycolor}{rgb}{0.30,0.3,0.30}

\definecolor{ForestGreen}{RGB}{34,139,34}
\definecolor{mauve}{rgb}{0.7,0,0.43}
\definecolor{dkgreen}{rgb}{0,0.6,0}
\definecolor{darkgreen}{rgb}{0,0.6,0}
\definecolor{darkorange}{rgb}{1.0, 0.55, 0.0}
\definecolor{lightblue}{rgb}{0,0.2,0.5}
\definecolor{blue1}{rgb}{0,0.1,0.9}

\definecolor{lightblue}{rgb}{0,0.2,0.5}

\usepackage{enumerate} 
\usepackage[numbers]{natbib} 
\usepackage{yfonts} 
\usepackage[title]{appendix} 

\usepackage{comment}
\usepackage[margin=0.25in]{geometry}
\usepackage{pgfplots,pgfplotstable}
\pgfplotsset{width=10cm,compat=1.9}
\usepackage[edges]{forest}

\usepackage{empheq} 

\usepackage{cases}

\usepackage{subcaption} 

\usepackage{float}

\newenvironment{Proofy}{\removelastskip\par\medskip
\noindent{\em Proof\hskip-0.1cm } \rm}{\penalty-20\null\hfill$\square$\par\medbreak}

\allowdisplaybreaks

\makeatletter
\let\@fnsymbol\@arabic 
\makeatother

\theoremstyle{plain}
\newtheorem{theorem}{Theorem}[section]
\newtheorem{proposition}[theorem]{Proposition}
\newtheorem{definition}[theorem]{Definition}
\newtheorem{corollary}[theorem]{Corollary}
\newtheorem{lemma}[theorem]{Lemma}
\newtheorem{remark}[theorem]{Remark}

\theoremstyle{remark}
\newtheorem{example}[theorem]{Example}

\theoremstyle{definition}

\newcommand{\R}{\mathbb R}

\newcommand{\D}{\mathcal D}

\def\P{\mathbb P}
\def\Q{\mathbb Q}
\def\/{\,|\,} 

\def\P{\mathbb{P}}

\makeatletter
\newcommand*\rel@kern[1]{\kern#1\dimexpr\macc@kerna}
\newcommand*\widebar[1]{%
  \begingroup
  \def\mathaccent##1##2{%
    \rel@kern{0.8}%
    \overline{\rel@kern{-0.8}\macc@nucleus\rel@kern{0.2}}%
    \rel@kern{-0.2}%
  }%
  \macc@depth\@ne
  \let\math@bgroup\@empty \let\math@egroup\macc@set@skewchar
  \mathsurround\z@ \frozen@everymath{\mathgroup\macc@group\relax}%
  \macc@set@skewchar\relax
  \let\mathaccentV\macc@nested@a
  \macc@nested@a\relax111{#1}%
  \endgroup
}
\makeatother

\makeatletter
\DeclareRobustCommand\widecheck[1]{{\mathpalette\@widecheck{#1}}}
\def\@widecheck#1#2{%
    \setbox\z@\hbox{\m@th$#1#2$}%
    \setbox\tw@\hbox{\m@th$#1%
       \widehat{%
          \vrule\@width\z@\@height\ht\z@
          \vrule\@height\z@\@width\wd\z@}$}%
    \dp\tw@-\ht\z@
    \@tempdima\ht\z@ \advance\@tempdima2\ht\tw@ \divide\@tempdima\thr@@
    \setbox\tw@\hbox{%
       \raise\@tempdima\hbox{\scalebox{1}[-1]{\lower\@tempdima\box
\tw@}}}%
    {\ooalign{\box\tw@ \cr \box\z@}}}
\makeatother

\newcommand{\RV}[1]{{\mathrm{RV}}_{{#1}}}

\newcommand{\SRV}[2]{{\mathrm{2RV}}_{{#1}, {#2}}}

\usepackage{bm}
\usepackage{leftindex}

\def\Pa{\mathbf{P}^{\alpha}}
\def\La{\mathscr{L}_{\alpha}}
\def\Da{\mathbf{D}_{\alpha}}

\DeclareMathOperator{\oD}{d}
\def\D{\oD\!}
\def\dif{\,\D}

\usepackage{tikz}
\def\R{{\mathord{\mathbb R}}} 
\def\Rp{\R_{+}}

\newcommand{\iid}{\textnormal{i.i.d.\ }}
\newcommand{\lbra}{[\![}
\newcommand{\rbra}{]\!]}
\newcommand{\esp}{\mathbb{E}}

\newcommand{\prob}{\mathbb{P}}
\newcommand{\ind}{\mathds{1}}
\newcommand{\bs}{\setminus}
\newcommand{\lb}{\{}
\newcommand{\rb}{\}}

\newcommand\numberthis{\addtocounter{equation}{1}\tag{\theequation}}

\numberwithin{equation}{section} 

\usepackage{hyphenat}
\usepackage{booktabs}
\usepackage{graphicx}

\usepackage[normalem]{ulem}
\usepackage{soul}

\begin{document}

\title{
\huge
{Convergence rates for the extreme value theorem {via} Stein's method} 
} 

\author{
 Bruno Costac\`eque\footnote{
 School of Physical and Mathematical Sciences, 
 Nanyang Technological University, 
 21 Nanyang Link, Singapore 637371.
 \href{mailto:bruno.costaceque@ntu.edu.sg}{bruno.costaceque@ntu.edu.sg}
 }
 \qquad
 Nicolas Privault\footnote{
 School of Physical and Mathematical Sciences, 
 Nanyang Technological University, 
 21 Nanyang Link, Singapore 637371.
 \href{mailto:nprivault@ntu.edu.sg}{nprivault@ntu.edu.sg}
}
\\
\small
Division of Mathematical Sciences
\\
\small
School of Physical and Mathematical Sciences
\\
\small
Nanyang Technological University
\\
\small
21 Nanyang Link, Singapore, Singapore 637371
}

\maketitle

\vspace{-0.5cm}

\begin{abstract}
  We derive convergence rates for the approximation of the
  Fréchet distribution $\mathcal{F}(\alpha)$ with parameter $\alpha > 0$ 
  by sequences of
  renormalized maxima in the extreme value theorem.
  Our proofs
  rely on the application of
  the infinitesimal generator approach to
  Stein's method to max-stable distributions,
  using the family of Markov semi-groups 
  recently introduced in \cite{CostacequePhD, Costaceque24}.
  We develop two different approaches to compute rates of convergence; 
  the first one relies on the second-order regular variation assumption,
  while the second one 
  requires the existence of a density function for the base distribution. 
  In particular, with the first approach, our bounds are expressed using the Kolmogorov
  distance, and the Wasserstein distance when $\alpha > 1$.
  The second approach allows also rates for a smooth H{\"o}lder distance
  when $\alpha \in (0,1)$. In both cases, we also obtain convergence rates for moments when
  they exist. 
\end{abstract}


\baselineskip0.7cm

\section{Introduction}
\noindent
 Let
 $(X_n)_{n\geq 1}$ denote a sequence of independent
 identically distributed (i.i.d.) random
 variables with common
 tail distribution function $\widebar{F}$.
 The
 Fisher--Tippett--Gnedenko theorem, or extreme value theorem
 \cite[Theorem~1.1.3]{deHaan07},
 states if the normalized maxima
$$
  Z_n\coloneq \frac{1}{a_n}
  ( \max ( X_1,\ldots , X_n) - d_n)
$$
  converge in distribution to a limit $H$
  for some normalizing real sequences
  $(a_n)_{n\geq 1}\subset (0,\infty)$
  and
  $(d_n)_{n\geq 1}\subset \R$,
  then $H$ must belong to one of the families
  of Fr\'echet, Weibull, or Gumbel distributions.
  In addition,
  from \textit{e.g.}
  \cite[Proposition~1.11]{Resnick87} or \cite[Theorem~1.2.1]{deHaan07},
  when $d_n=0$, $n\geq 1$,
  convergence holds to the
  Fr\'echet distribution
  $\mathcal{F}(\alpha)$
  with parameter $\alpha > 0$ and
  distribution function $\Phi_{\alpha}(x) = \exp(-x^{-\alpha})$, $x > 0$,
  if and only if
  $\widebar{F}$
  satisfies the \textit{first-order regular variation} condition
\begin{align}\label{Intro_FO}
\underset{t \to \infty}{\lim} \frac{\widebar{F}(tx)}{\widebar{F}(t)} = x^{-\alpha},\quad x > 0,
\end{align}
 which is denoted $\RV{-\alpha}$ in what follows.
 In this case
$(a_n)_{n \geq 1}$ is given by $a_n = F^{\leftarrow}(1 - n^{-1})$,
where
\[
F^{\leftarrow}(y) \coloneq \inf\{ x \in \R,\ F(x) \ge y\}, \quad y\in (0,1),
\]
 denotes the left-continuous generalized
 inverse of $F$.
 This convergence result
 accounts for the central role played by the Fr\'echet
 distribution,
 as it shows that a renormalized record of \iid data
 with regularly varying tails can
 converge to no other law.

 \medskip

 In applications, however, this qualitative statement may not be sufficient.
 For instance, to estimate the tail index $\alpha$, the block maxima method,
 see \textit{e.g.} \cite{Embrechts13},
 relies on the fact that, as the number of observations $n$
tends to infinity, the normalized block maxima
\[
a_n^{-1}M_{k,n} = a_n^{-1}\max_{(k-1)n \leq i \leq kn} X_i, \quad
 k =1,2, \ldots,
\]
approximately follows the Fr\'echet distribution $\mathcal{F}(\alpha)$, with $n$ the length of each block.
One can then estimate $\alpha$ by treating $M_{1,n}, M_{2,n}, \ldots$
 as samples of \iid
Fr\'echet random variables. The accuracy of this procedure, though, depends crucially on how
much data is needed for the law of $a_n^{-1}M_{k,n}$ to be well approximated by the Fr\'echet distribution as $n$ tends to infinity -- that is, on the \emph{speed} of convergence in the extreme value theorem.

\medskip

  The derivation of quantitative convergence bounds
  usually relies on higher order information on
  the distribution of the $X_i$s.
  In particular, the second-order regular variation condition
  states that there exists a function $A$ of ultimately
 constant sign, and such that  
 \begin{align}
   \nonumber 
 \frac{1}{A(t)}\Big(\frac{\widebar{F}(tx)}{\widebar{F}(t)} - \frac{1}{x^{\alpha}}\Big)
 \end{align}
 converges pointwise to a non-constant function
 as $x$ tends to infinity. The function $A$ is called
 the \textit{auxiliary function},
 and contains information about the speed of convergence 
 to extreme value distributions, see, 
 \textit{e.g.}, \cite{Smith82}.

\medskip
 
 In the one-dimensional setting, the question of
 convergence speed has been addressed
\textit{e.g.} 
in \cite{Cohen82, Smith82, deHaan96, Cheng01} 
 under a variety of conditions
  (second-order regular variation,
 second-order von Mises, 
 slow variation with remainder, etc).
  For example, \cite[Theorem 3.1]{deHaan96},
 yields an estimate on the Kolmogorov distance
 $\mathrm{d}_{K}(Z_n, Z)$ 
 to the Fr\'echet distribution $\mathcal{F}(\alpha)$,
 given in terms of the
 second-order von Mises property 
 of the function $(-1/\log F)^{\leftarrow}$. 

 \medskip

 On the other hand, Stein's method, introduced in
\cite{Stein72} as an alternative proof of the Berry--Esseen theorem,
 offers a more systematic approach to quantitative convergence
 estimates, by providing general bounds
on the distance between a fixed target distribution (\textit{e.g.} the normal distribution) and
an approaching one (\textit{e.g.} the law of a renormalized sum of \iid square-integrable random
variables).
 In this setting, the underlying probability metric is not restricted to the Kolmogorov distance, and
can be expressed as an \textit{integral probability metric}
\[
\mathrm{d}(\prob_{X}, \prob_{Y}) = \underset{h \in \mathcal{H}}{\sup}\ \vert \esp[h(X)] - \esp[h(Y)] \vert,
\]
 where $\mathcal{H}$ denotes a certain class of functions.
This marks an important difference with the aforementioned approaches
which either apply to the Kolmogorov distance, 
or require significant additional effort to
be adapted to other metrics.
 However, exact rates of convergence and Edgeworth
expansions, such as the one given in \cite{Cheng01}, tend to be considerably harder to obtain \textit{via}
Stein's method than through more \textit{ad hoc} approaches.

\medskip

 Stein's method can be quickly summarized
 as follows.
 Given a target distribution $\nu$ and a family of
test functions $\mathcal{H}$ and the corresponding \textit{integral probability metric}:
\begin{align}\label{Intro_IPM}
\mathrm{d}_{\mathcal{H}}(\mu, \nu) \coloneq \underset{h \in \mathcal{H}}{\sup}\ \vert\esp_{\mu}[h] - \esp_{\nu}[h]\vert,
\end{align}
one looks for a linear operator $\mathscr{L}$ characterizing $\mu$
in the sense that, for any probability measure $\mu$,
the equality $\mu = \nu$ can be characterized by the condition
\begin{align}
  \nonumber 
 \esp_{\mu}[\mathscr{L}g] = 0\ \mathrm{for\ all\ } g \in \mathcal{G},
\end{align}
where $\mathcal{G}$ is an appropriate class of functions depending on $\mathcal{H}$. The next step consists in finding for any test-function $h$ a solution $g_{h}$ to the so-called \textit{Stein equation}
\begin{align}\label{Intro_Stein_eq}
\mathscr{L}g_{h} = h - \esp_{\nu}[h].
\end{align}
The regularity of the \textit{Stein solution} $g_{h}$ (boundedness, derivability, \textit{etc.}) hinges strongly on the regularity of $h$ itself. Less smooth choices of $h$ correspond to common metrics (Kolmogorov, total variation, Radon, \textit{etc.}) and tend to lead to intractable computations.
 Using \eqref{Intro_Stein_eq}, one can rewrite \eqref{Intro_IPM} as
\[
\mathrm{d}_{\mathcal{H}}(\mu, \nu) = \underset{h \in \mathcal{H}}{\sup}\ \vert \esp_{\mu}[\mathscr{L}g_{h}]\vert, 
\]
which turns out to be easier
to bound than \eqref{Intro_IPM},
provided that one can control the Stein solution $g_{h}$ for $h \in \mathcal{H}$,
 see \textit{e.g.} \cite{Decreusefond15, Ley17, Anastasiou21} and references therein for details.

\medskip

 Historically, Stein's method was first applied to limit theorems involving the normal
distribution, such as the Berry--Esseen theorem, and since then it has
 been extended to many other
 distributions (Poisson, $\alpha$-stable, Beta, Gamma, \textit{etc.}).
 Previous applications of the Stein
 method in extreme
 value theory include
 \cite{Smith87} and \cite{Holst90}
 who used the Stein--Chen method to bound
the distance between the point process of exceedances above a high threshold and a genuine
Poisson process, see also \cite{Feidt13} for an
extension to the multivariate setting.
These
works do not rely on an operator characterizing an extreme value distribution
 (EVD)
 itself, but rather
on the classical correspondence between extreme value theory and Poisson process approximation,
 and the first works to use
 Stein operator to characterize an extreme
 value distribution
  are \cite{Bartholome13} and
  \cite{Kusumoto20},
  with the derivation of convergence rates for Pareto,
  exponential, and uniform base distributions.

\medskip

 We also note that
 bounds on the Kolmogorov and Wasserstein distances between
 a maximum of \iid random
 variables and the Fr\'echet distribution
 have been independently obtained in \cite{Mansanarez25},
 following a different approach
 based on Stein's discrepancy:
\[
\Delta(\Q, \P) \coloneq \int_{c_{\Q}}^\infty \Big\vert 1 - \frac{r_{\P}(x)}{r_{\Q}(x)}\Big\vert q(x) \dif x,
\]
where $q$ is the density of $\Q$ and $r_{F}$ denotes the \textit{reverse hazard rate function}
\[
r_{F}(x) \coloneq (\log F)'(x)\ind_{[c_F, +\infty)}(x),
\]
with $c_F$ the lower bound of the support of $F$. When $c_F \leq 0$, this yields the
discrepancy between the distribution of $Z_n$ and $\mathcal{F}(\alpha)$ the explicit expression
\begin{equation}
\nonumber 
\int_0^\infty  \Big\vert 1 - na_n \frac{x^{\alpha + 1}}{\alpha}r_{F}(a_nx)\Big\vert\phi_{\alpha}(x) \dif x,
\end{equation}
where $\phi_{\alpha}$ is the density of $\mathcal{F}(\alpha)$.
This local, density-ratio Stein operator -- adapted from the comparison-of-generators framework of \cite{Ley17} -- is distinct from the nonlocal generator
 $\mathscr{L}_\alpha$ used below, see~\eqref{djkla1}.
The general result of \cite{Mansanarez25} 
only guarantees that this discrepancy vanishes, and explicit rates are
obtained on a case by case basis, with a sharp, exact constant available only in the Pareto case. 
\medskip

 In this paper, we apply
the generator
 approach to Stein's method
 to extreme value distributions 
 using a new family of Markov
 semi-groups akin to the Ornstein--Uhlenbeck semi-group,
 recently introduced in \cite{CostacequePhD, Costaceque24}.
   Applications of those semi-groups
   to the derivation
 of convergence rates for the coupon collector problem and
 to the de Haan--LePage
 series in higher dimensions
 have been developed in \cite{Costaceque24_cpn}
 and
 \cite{Costaceque24}.
 In particular, we extend the results obtained therein
 \textit{via} two complementary approaches
 to the derivation of explicit convergence rates
 in the extreme value theorem.

\medskip

 The first
 approach, developed in Section~\ref{s3},
 applies the Stein equation~\eqref{djkla1}
 under a second-order regular variation assumption
 with auxiliary function~$A$
 on the tail distribution function
 $\widebar{F}$.
 Our first approach will yield rates of convergence in the Kolmogorov and 
 Wasserstein distances, as well as for moments,
 and makes precise the role played by
 the Pareto distribution in the Fr\'echet generator
 $\La$. 
 In Corollary~\ref{SRV_Wass}, we obtain
Wasserstein bounds of the form
\[
\mathrm{d}_{W} (Z_{n}, \mathcal{F}(\alpha)) \leq C\Big(\frac{1}{n} + A(a_n)\Big),
 \quad n \geq 1. 
\]

 The second approach, developed in Section~\ref{s4},
 replaces the second-order assumption
 with the existence of a density of the
 random variables $X_i$.
 In Theorem~\ref{Cor_Int_all} we derive bounds of the
 form
  \begin{align*}
      \nonumber 
      \mathrm{d}_{W} (Z_n, \mathcal{F}(\alpha) ) \leq \esp [\vert Z_n \vert^{p} ]^{{1}/{p}}F_{Z_n}(K)^{{1}/{q}}  +
      C_\alpha
      F_{Z_n}(K)+ \alpha\int_{K}^\infty  rF_{Z_n}(r)\bigg\vert\frac{f_{Z_n}(r)}{F_{Z_n}(r)} - \frac{\alpha}{r^{\alpha + 1}}\bigg\vert \dif r,\numberthis
\end{align*}
  where $q = p/(p-1)$ denotes the conjugate of $p$, 
  where $f_{Z_n}$ denotes the probability density function
  of $Z_n$.
  Condition~\eqref{4.6} 
    below can be easily verified
    on standard heavy-tailed families
    of distributions,
    and precise convergence rates are obtained in a number of
  examples including
  Hall–Weiss,
  log-gamma,
  log-logistic,
  Student and
  generalized Beta prime distributions.
  This improves on the results
    of \cite{Mansanarez25}
    by providing general explicit rates,  
    together with bounds for moments
    and
    for the
    distance $\mathrm{d}_{W_b}$
    defined using H\"older-type functionals.  

 \medskip

 The rest of the paper is organized as follows.
  Section~\ref{s2} recalls the definitions and notation
used throughout, together with the generator approach to Stein's method for extreme value
distributions and the necessary background on first- and second-order regular variation.
Section~\ref{s3} develops our first approach, based on the second-order regular variation assumption,
and derives the corresponding rates of convergence in Kolmogorov and Wasserstein distance, as
well as for moments. Section~\ref{s4} develops the second approach, which compares the intensity
function of the normalized maximum to that of the Fr\'echet distribution under an absolute
continuity assumption. It also provides rates of convergence for the same functionals as in the previous section, as well as for Hölder functions.

\section{Preliminaries and notation} 
\label{s2}

\subsection{Integral probability metrics}
\noindent
 By an integral probability metric, we mean
a metric on a space of probability measures over some set $\mathcal{X}$, 
 that takes the form:
 \begin{equation}
   \label{jklsda1} 
\mathrm{d}(\prob_{X}, \prob_{Y}) = \underset{h \in \mathcal{H}}{\sup}\ \vert \esp[h(X)] - \esp[h(Y)] \vert
\end{equation} 
for some space $\mathcal{H}$ of test functions $h : \mathcal{X} \to \R$, and $\prob_{X}$, $\prob_{Y}$ the laws of two random variables $X$, $Y$ of $\mathcal{X}$. We will often abuse notations and write more concisely $\mathrm{d}(X, Y)$ for $\mathrm{d}(\prob_{X}, \prob_{Y})$.
When $\mathcal{H} := \{ h_{z} := \ind_{(-\infty, z]} \ : \ z\in \R\}$, 
\eqref{jklsda1} recovers the Kolmogorov distance 
\[
\mathrm{d}_{K}(X, Y) \coloneq \underset{z \in \R}{\sup}\ \vert \prob(X \leq z) - \prob(Y \leq Z) \vert.
\]
If $\mathcal{H}$ is the set $\mathrm{Lip}_{1}(\R)$ of $1$-Lipschitz functions $h : \R \to \R$ such that
\[
\vert h(x) - h(y)\vert \leq \vert x - y \vert,
\]
then \eqref{jklsda1}
yields the $1$-Wasserstein distance $\mathrm{d}_{W}$,
which will be used when $\alpha > 1$
for approximation of the Fréchet distribution $\mathcal{F}(\alpha)$.
 When $\alpha \in (0,1]$,
   because $\mathcal{F}(\alpha)$ does not possess a first moment, 
   we will follow \cite{Chen24} and \cite{Coutin24}
   and replace the use of
   the Wasserstein distance
   with that of the distance $\mathrm{d}_{W_b}$
   defined from the set
\[
\mathcal{H}_b \coloneq \big\{
h : \Rp \to \R,\ \vert h(x) - h(y) \vert \leq \min\big(
\vert x - y \vert, \vert x - y \vert^b \big) \big\}
\]
 of Lipschitz test functions which are also $b$-Hölder
 for some $b \in (0,\alpha)$. 
 Here, the Lipschitz property of $h \in \mathcal{H}_b$ ensures that $h$ is absolutely continuous, with a bounded derivative, while the Hölder property guarantees that $\esp[\vert h(Z) \vert]< \infty$.  

\subsection{Generator approach to Stein's method for EVDs}
\noindent
In what follows, for any $x,y \in \R$ we let $x \oplus y = \max ( x,y)$,
and we let $\widebar{F} = 1 - F$
denote the tail distribution function of
any probability distribution function $F$. 
 Given $\alpha > 0$,
 let $\mathscr{F}_{\alpha}$ be the set of absolutely continuous functions $g : \Rp \to \R$ such that
\[
x\vert g'(x) \vert \leq C x^{\beta},\quad x \in \Rp, 
\]
for some $C > 0$ and $\beta \in (-\infty, \alpha)$.
For $g \in \mathscr{F}_{\alpha}$, we consider 
 the non-local operator  $\Da$ defined as
\[
\Da g(x) \coloneq \alpha \int_0^\infty (g(x \oplus y) - g(x))\frac{\dif y}{y^{\alpha + 1}},\quad x \in \Rp,  
\]
 and let 
\[
\La g(x) \coloneq -\frac{1}{\alpha}xg'(x) + \Da g(x),\quad x \in \Rp. 
\]
\begin{remark}
 Using integration by parts, we can write 
\begin{align}
  \nonumber 
  \Da g(x) = \int_{x}^\infty g'(y) \frac{\dif y}{y^{\alpha}} =
  \int_{x}^\infty yg'(y)\frac{\dif y}{y^{\alpha+1}}, 
\end{align}
 hence 
  \begin{align}
      \label{Da_alt_2}
\Da g(x) = \frac{1}{\alpha x^{\alpha}} \esp[xYg'(xY)],\quad x > 0, 
\end{align}
  where $Y$ is a Pareto $\mathcal{VP}(\alpha)$ random variable,
  with tail distribution function $y^{-\alpha}$ on $[1, +\infty)$. 
\end{remark}
\noindent 
Noting that
 when $Z$ is a Fréchet $\mathcal{F}(\alpha)$ random variable,
 \textit{i.e.} $Z \sim \mathcal{F}(\alpha)$,
 the random variable $h(Z)$ has a finite expectation
 for any $h \in \mathscr{F}_{\alpha}$,
 we define the operator
\[
\Pa_{t}h(x) \coloneq \esp [h(e^{-{t}/{\alpha}}x \oplus (1-e^{-t})^{{1}/{\alpha}}Z)],\quad t, x \geq 0. 
\]
It can be proved,
see \textit{e.g.} \cite{Costaceque24}, 
that $(\Pa_{t})_{t \ge 0}$ defines a Markov semi-group with generator $\La$,
\textit{i.e.} for any $g \in \mathscr{F}_{\alpha}$ we have 
\[
\La g(x) = \underset{t \to 0^{+}}{\lim} \frac{1}{t}(\Pa_{t}g(x) - g(x)),\quad x \ge 0.
\]
Those operators can be used to bound differences of expectations with respect to some positive random variable $W$ and a Fréchet random variable $Z$ by writing
\[
\esp[h(W)] - \esp[h(Z)] = \esp [\La g_{h}(W)],\quad h \in \mathscr{F}_{\alpha}, 
\]
where $g_{h}$ is a Stein solution for $\La$, \textit{i.e.} a solution to the Stein equation
\begin{equation}
  \label{djkla1}
\La g_{h}(x) = h(x) - \esp[h(Z)],\quad x > 0.
\end{equation}
Since $\La$ is the generator of the Markov semi-group $(\Pa_{t})_{t \ge 0}$,
 \eqref{djkla1} 
 admits a solution expressed as 
\begin{align}
    \nonumber 
g_{h}(x) = -\int_0^\infty (\Pa_{t}h(x) - \esp[h(Z)]) \dif t.
\end{align}
Under the condition
 $h \in \mathscr{F}_{\alpha}$, we can differentiate $\Pa_{t}h$ as 
\begin{align}
    \nonumber 
    (\Pa_{t}h)'(x) = e^{-{t}/{\alpha}}e^{-{\gamma_{t}}{x^{-\alpha}}}
    h'(e^{-{t}/{\alpha}}x),
\end{align}
where we let $\gamma_{t} \coloneq e^{t} - 1$.
Hence, $g_{h}$ is also differentiable, with 
\begin{align}\label{Stein_sol_diff}
g'_{h}(x) = -\int_0^\infty  e^{-{t}/{\alpha}}e^{-{\gamma_{t}}{x^{-\alpha}}}h'(e^{-{t}/{\alpha}}x) \dif t.
\end{align}
 Although
 $h_{z} := \ind_{(-\infty, z]}
  $,
  $z>0$,
  is not absolutely continuous and thus does not belong to $\mathscr{F}_{\alpha}$, we note that the corresponding Stein solution $g_{z}$ is differentiable everywhere, with derivative given
    for $x\not= z$ by 
\begin{align}\label{g_z_diff}
  g_{z}'(x) = \frac{\alpha}{x}e^{
    x^{-\alpha} - z^{-\alpha}}\ind_{(z, +\infty)}(x). 
\end{align}
\subsection{First and second-order regular variation}
\noindent 
Throughout this paper, we will use the so-called Potter bounds on 
 any tail distribution
function $\widebar{F} : (0,\infty) \to \R$ 
satisfying the regular variation condition \eqref{Intro_FO},
see \cite[Proposition~B.1.9]{deHaan07}.
Namely, for every $\delta, \varepsilon > 0$, there exists $t_0 > 0$ such that for every $t > t_0$ we have
the inequalities 
\begin{align}
\label{1_Potter}
    \frac{1 - \varepsilon}{x^{\alpha}}
    \min (x^{-\delta}, x^{\delta}) \leq \frac{\widebar{F}(tx)}{\widebar{F}(t)} \le
    \frac{1 + \varepsilon}{x^{\alpha}}\max (x^{-\delta}, x^{\delta}),
    \quad
    x > \frac{t_0}{t}.
    \end{align}
Definition~\ref{djklaa} provides
 a way to quantify the first-order regular variation property. 
\begin{definition}
  \label{djklaa}
  Let $\gamma \in \R$.
  A function
  $h \in \RV{\gamma}$ is said to be of \textit{second-order regular variation at infinity} with first-order parameter $\gamma \in \R$ and second-order parameter $\rho \leq 0$, denoted $h \in \SRV{\gamma}{\rho}$, if there exists
  an ultimately positive or negative function $A : \R \to \R$ such that 
  for any $x > 0$ we have 
  \begin{align}
    \nonumber 
    \underset{t \to \infty}{\lim} \frac{1}{A(t)}\Big(\frac{h(tx)}{h(t)} - x^{\gamma}\Big) = 
    \left\{
    \begin{array}{ll}
      x^{\gamma}
      \displaystyle \frac{x^{\rho} - 1}{\rho}
      \quad & \mbox{when } \rho < 0, 
      \smallskip
      \\
      x^{\gamma}
      \log x
      \quad & \mbox{when } \rho = 0. 
      \end{array} 
\right. 
\end{align}
The function~$A$ is commonly called \textit{auxiliary function}.
\end{definition}
\noindent
Tail distribution functions $\widebar{F}$ in
$\SRV{\gamma}{\rho}$ admit more explicit
descriptions when $\rho < 0$.
By \cite[Lemma~3]{Hua11}, $\widebar{F}
 \in \SRV{\gamma}{\rho}$ can be expressed as 
\[
\widebar{F}(x) = \ell\frac{L(x)}{x^{\alpha}},
\]
where $\ell > 0$ and $L$ is a function
such that $\lim_{t\to \infty} L(t) = 1$
and $\vert 1 - L \vert \in \RV{\rho}$.
As a consequence,
 for any $k \in (0,\ell)$ and $K \in (\ell, +\infty)$, there exists $x_{0}$ such that 
\begin{align}\label{2_Bnd_tail}
  \frac{k}{x^{\alpha}} \leq \widebar{F}(x) \leq \frac{K}{x^{\alpha}},
   \quad x \ge x_{0}.
\end{align}
In addition, 
\cite[Proposition~4]{Hua11}, 
provides a second-order version of Potter's bounds
\eqref{1_Potter},
 \textit{i.e.} for any $\delta, \varepsilon > 0$, there exists $t_0 = t_0 (\delta, \varepsilon)$ such that for every $t > t_0$ we have
\begin{align*}\label{2_SORV_ineq}
x^{-\alpha} \frac{x^{\rho} - 1}{\rho} & - \varepsilon x^{-\alpha + \rho}\max (x^{\delta}, x^{-\delta})\\
&\leq \frac{1}{A_{0}(t)}\Bigg(\frac{\widebar{F}(tx)}{\widebar{F}(t)} - x^{-\alpha}\Bigg) \leq x^{-\alpha} \frac{x^{\rho} - 1}{\rho} + \varepsilon x^{-\alpha + \rho}\max (x^{\delta}, x^{-\delta}),
\quad
x > \frac{t_0}{t},
\numberthis
\end{align*}
where $A_{0}(t) \coloneq -\rho(1 - L(t))/L(t)$. Moreover,
$A_{0}(t)$ can also be taken to be the auxiliary function~$A(t)$,
which is the choice made in the sequel.
The following basic lemma is a useful consequence of the second-order regular variation property. 
\begin{lemma}
  \label{djkla1a}
   For any tail distribution function
  $\widebar{F}$ in $\SRV{\alpha}{\rho}$ with $\alpha > 0$, $\rho < 0$,
  and auxiliary function~$A$, we have 
\begin{align}\label{o(An)}
n\widebar{F}(a_n) - 1 = o(A(a_n)),\ \text{as\ } n\ {tends\ to\ } \infty 
\end{align}
 where 
 $a_n \coloneq U^{\leftarrow}(n)$, $n\geq 1$,
 and
 $U(t) \coloneq 1/\widebar{F}(t)$. 
\end{lemma}
\begin{proof}
  By definition of the left-continuous inverse
  $U^{\leftarrow}$,
   for every $\varepsilon > 0$, we have the inequalities:
\begin{align*}
\frac{U((1-\varepsilon)U^{\leftarrow}(n))}{U(U^{\leftarrow}(n))} \leq \frac{n}{U(U^{\leftarrow}(n))} = n\widebar{F}(a_n) \leq \frac{U((1 +\varepsilon)U^{\leftarrow}(n))}{U(U^{\leftarrow}(n))}\cdotp
\end{align*}
Taking $\varepsilon 
 \coloneq A(a_n)$ yields 
\begin{align}\label{Resnick_U}
  \frac{1}{A(a_n)} \Bigg( \frac{U\big((1-A(a_n) )a_n\big)}{U(a_n)} - 1\Bigg) \leq
  \frac{n\widebar{F}(a_n) - 1}{A(a_n)} \leq \frac{1}{A(a_n)} \Bigg(\frac{U\big((1+A(a_n) )a_n\big)}{U(a_n)} - 1\Bigg)
\end{align}
 Since $\widebar{F} \in \SRV{\alpha}{\rho}$, 
 we have the asymptotic expansion 
 \begin{align}
   \label{djka} 
  \frac{U(t)}{U(tx)} = \frac{\widebar{F}(tx)}{\widebar{F}(t)}
       {=} x^{-\alpha} + x^{-\alpha}
       \frac{x^{\rho} - 1}{\rho}
       A(t) + o(A(t)) 
	\end{align} 
as $t$ tends to infinity,
locally uniformly in $x$ on $(0, \infty)$,
see \cite[p.~388]{deHaan07}.
Applying \eqref{djka} to $t = a_n$ and $x = 1 \pm A(a_n)$ shows
that the middle term of \eqref{Resnick_U} vanishes 
since $((1\pm A(a_n) )^{\rho} - 1)/\rho$ tends to zero
at the rate $A(a_n)$ as $n$ goes to infinity,
 proving \eqref{o(An)}. 
\end{proof}

\noindent
In what follows, to shorten notations,
 for all $t, x > 0$ we will write 
\[
\widebar{F}_{t}(x) \coloneq \frac{\widebar{F}(tx)}{\widebar{F}(t)}\cdotp
\] 
We observe that for $x \ge 1$
we have $\widebar{F}_{t}(x) = \prob(X / t > x \/ X/t > 1)$, so that
the restriction of $\widebar{F}_{t}$ $[1, +\infty)$ is the tail
  distribution function of $X/t$ conditionally
  to $\lb X/t > 1 \rb$. In particular, the regular variation
   Assumption~\eqref{Intro_FO} implies that $X/t$
  converges in distribution to the Pareto distribution $\mathcal{VP}(\alpha)$
   as $t$ tends to infinity. 
\section{Bounds using the second-order regular variation}
\label{s3}
\subsection{Main results}
\noindent
 Next is the main result of this section,
which will be applied to the special cases of the Kolmogorov and Wasserstein distances as well as power test functions.
\begin{theorem}\label{main_theorem}
  Let $(X_{i})_{i \ge 1}$ be a sequence of \iid random variables with common distribution function $F$. Assume that
  $\widebar{F} \in \SRV{\alpha}{\rho}$, with $\alpha > 0$, $\rho < 0$,
  and auxiliary function~$A$.
  Set
  $$Z_n \coloneq a_n^{-1}
  \max(X_{1},\dots,X_{n}),
  $$
  where   
  $a_n \coloneq F^{\leftarrow}(1 - 1/n)$,
  $n\geq 1$. Let $\mathcal{H}$ be a set of functions $h : \R \to \R$ such that the associated Stein solution $g_{h}$ exists, belongs to $\mathcal{C}^{1}((0,\infty))$, where $x \mapsto xg_{h}'(x)$ is absolutely continuous and satisfies 
\begin{align}\label{H_1}
x\vert g_{h}'(x) \vert \leq C_{1}x^{\beta} 
\end{align}
 and 
\begin{align}\label{H_2}
\vert (xg_{h}'(x))'\vert \leq C_{2}x^{\gamma} 
\end{align}
for almost every $x >0 $, some $\beta \in (-\infty, \alpha),\ \gamma \in (-\infty, \alpha - 1)$,
 $C_{1}, C_{2} > 0$ independent of $h$. Then,
 there exists $C_{\alpha} > 0$ such that 
\begin{align}\label{main_theorem_ineq}
  \underset{h \in \mathcal{H}}{\sup}\ \big\vert \esp[h(Z_{n})] - \esp[h(Z)]\big\vert \leq C_{\alpha}\Big(\frac{1}{n} + A(a_n)\Big) + \underset{h \in \mathcal{H}}{\sup}\ \big\vert\esp\big[h(Z_{n})\ind_{U^{c}_{n}}\big] - \esp[h(Z)]\prob(U_{n}^{c})\big\vert,
   \quad n \geq 1.
\end{align}
\end{theorem}
\noindent 
The next two results are direct consequences of Theorem~\ref{main_theorem}. 
\begin{corollary}
\label{SRV_Wass}
Wasserstein bounds.
Let $(X_{i})_{i \ge 1}$ be a sequence of \iid random variables with common distribution function $F$. Assume that $\widebar{F} \in \SRV{\alpha}{\rho}$, with $\alpha > 1$, $\rho < 0$, and auxiliary function~$A$. There exists a positive constant $C$ such that
\begin{equation}
  \label{fkjda} 
\mathrm{d}_{W} (Z_{n}, \mathcal{F}(\alpha)) \leq C\Big(\frac{1}{n} + A(a_n)\Big),
 \quad n \geq 1, 
\end{equation} 
if the $X_{i}$s are continuous random variables.
 If not, then for any $\delta \in (1, \alpha)$, there exists a constant $C_{\delta} > 0$ such that 
\[
\mathrm{d}_{W} (Z_{n}, \mathcal{F}(\alpha)) \leq C_{\delta}\Big(\frac{1}{n} + A(a_n)\Big)^{1 - {1}/{\delta}}, \quad n \geq 1. 
\]
\end{corollary}

\begin{proof}
  If $h$ is $1$-Lipschitz, then $g_{h}$ is continuously differentiable and satisfies assumptions \eqref{H_1} and \eqref{H_2} with $\beta = 1 < \alpha$ and $\gamma = 0 < \alpha - 1$, as well as $C_{1} = \alpha$ and $C_{2} = 2\alpha$ respectively,
  which completes the proof if $F$ is continuous.
  If $F$ is not continuous,
  we need to deal with the second term in the right-hand side of \eqref{main_theorem_ineq}, in which $\mathcal{H}$ is the set of $1$-Lipschitz functions from $(0,\infty)$ to $\R$. For this, we note that if
  $Z_{n}$ and $Z$ independent
  random variables on the same probability space,
 applying Hölder's inequality to $\delta \in (1,\alpha)$ yields 
\begin{align*}
  \underset{h \in \mathcal{H}}{\sup}\ \vert\esp [h(Z_{n})\ind_{U^{c}_{n}}] - \esp[h(Z)]\prob(U_{n}^{c}) \vert &\leq \esp [\vert Z_n - Z \vert\ind_{U^{c}_{n}} ]
  \\
	&\leq \esp [\vert Z_n \vert\ind_{U^{c}_{n}}] + \esp[Z]\prob(U_{n}^{c})\\
	&\leq \esp [\vert Z_n \vert^{\delta} ]^{{1}/{\delta}}\prob(U_{n}^{c})^{1 - {1}/{\delta}} + \Gamma\Big(1 - \frac{1}{\alpha}\Big)\prob(U_{n}^{c}),
\end{align*}
 where $\Gamma$ denotes the Gamma function.
\end{proof}
\noindent 
Next, we bound the speed of convergence of the moments of $(Z_{n})_{n \ge 1}$ to their limits when the $X_{i}$ are assumed to take nonnegative values.

\begin{corollary}
    Moment bounds.
     Let $(X_{i})_{i \ge 1}$ be a sequence of nonnegative \iid random variables with common distribution function $F$. Assume that $\widebar{F} \in \SRV{\alpha}{\rho}$, with $\alpha > 0$, $\rho < 0$ and auxiliary function~$A$.
    Then, there exists $C>0$ such that for every $b \in (0,\alpha)$
 we have 
 \begin{equation}
  \label{jkla1a} 
	\vert \esp[Z_{n}^{b}] - \esp[Z^{b}] \vert \leq C\Big(\frac{1}{n} + A(a_n)\Big),
 \quad n \geq 1, 
\end{equation}
if the $X_{i}$ are continuous random variables.
 If not, then for every $b \in (0, \alpha)$ and
 $\delta \in (1, \alpha / b )$, there exists a constant $C_{\delta}$ such that 
\[
\vert \esp[Z_{n}^{b}] - \esp[Z^{b}]\vert \leq C_{\delta}\Big(\frac{1}{n} + A(a_n)\Big)^{1 - {1}/{\delta}}, \quad n \geq 1.
\]
\end{corollary}

\begin{proof}
  We apply Theorem \ref{main_theorem} to $\mathcal{H} = \lb h_{b}:x \mapsto x^{b} \rb$. For those functions, we can take $\beta = b < \alpha$, $\gamma = b-1 < \alpha - 1$, as well as $C_{1} = \alpha$ and $C_{2} = 2\alpha(\alpha + 1)$ thanks to Lemma \ref{Lemma_g_b} below:
  \begin{align*}
  \vert(x g_{b}'(x))'\vert \le \vert g_b'(x) \vert + x\vert g''_b(x)\vert &\le \alpha x^{b-1} + \alpha |b-1|x^{b-1} + \alpha bx^{b - 1}\min(x^\alpha, x^{-\alpha})\\
		&\le \alpha x^{b-1} + \alpha(\alpha + 1) x^{b-1} + \alpha^{2}x^{b - 1}\min(x^\alpha, x^{-\alpha})\\
		&\le 2\alpha(\alpha + 1)x^{b-1}.
  \end{align*}
   The rest
  of the proof is the same as in the proof of
   Corollary~\ref{SRV_Wass}.
\end{proof}
\noindent 
Finally, we deal with the case of the Kolmogorov distance.
\begin{theorem}
	 \label{SRV_Kol}
    Let $(X_{i})_{i \ge 1}$ be a sequence of continuous \iid random variables with common distribution function $F$. Assume $\widebar{F} \in \SRV{\alpha}{\rho}$, with $\alpha > 0$, $\rho < 0$ and auxiliary function~$A$. Then
    there exists a positive constant $C$ such that for every $n \geq 1$
    we have 
    \begin{equation}
      \label{ajkla1} 
\mathrm{d}_{K} (a_n^{-1}M_{n}, \mathcal{F}(\alpha) ) \leq C\Big(\frac{1}{n} + A(a_n)\Big). 
\end{equation}
\end{theorem}

\begin{example}\label{Ex_Hall-Weiss}
\begin{enumerate}

\item {Continuous Hall-Weiss distribution}. One of the simplest examples of second-order regularly varying distributions is the Hall-Weiss distribution, defined as
Let $\alpha, \beta > 0$, and 
\[
\widebar{F}(t) \coloneq \frac{1}{
  2t^{\alpha}}\Big(1 + \frac{1}{
  t^{\beta}}\Big),
\quad t \ge 1. 
\]
 This tail function belongs to $\SRV{-\alpha}{-\beta}$, hence
 \eqref{fkjda},
 \eqref{jkla1a},
 \eqref{ajkla1}
 respectively
 yield 
 a convergence rate of order $n^{-1} + n^{-\beta/\alpha}$ in
 Wasserstein distance with $\alpha > 1$, 
 for the $b$-th moment with $b \in (0,\alpha)$,
 and in Kolmogorov distance.
 We note that \eqref{ajkla1}
 recovers the
 Kolmogorov distance rate obtained in
 \cite[Theorem 1]{Smith82},
 where it is proved that if $F$ is a distribution function such that $-\log F$ belongs to $\RV{-\alpha}$ and $L : t \mapsto -t^{\alpha}\log F(t)$ is slowly varying with remainder $g$, \textit{i.e.}
\[
\frac{L(tx)}{L(t)} - 1 = O (g(t)),\quad t \to \infty,
\]
where $g$ vanishes at infinity and satisfies 
\[
\frac{B}{x^{\theta}} \leq \frac{g(tx)}{g(t)} \leq C
\] 
 for some $B, C, \theta, t_0 > 0$ and 
 every $x \ge 1$ and $t \ge t_0$, then we have 
\[
\underset{r \in \Rp}{\sup}\ \vert F^{n}(a_{n}r) - \prob(Z \leq r) \vert = O(g(a_n)),
\]
where $a_n \coloneq F^{\leftarrow}(1 - 1/n)$ and $Z \sim \mathcal{F}(\alpha)$. If $\widebar{F}$ belongs to the Hall--Weiss class, additional computations show that $F = 1 - \widebar{F}$ satisfies the required assumptions
with $g(t) \coloneq t^{-\min(\alpha, \beta)}$. Since $a_n = O(n^{1/\alpha})$, this confirms that our rate of convergence coincides with the optimal one given in
\cite{Smith82}. 

\item {Discrete Hall-Weiss distribution}. Let $X$ be a random variable with the Hall-Weiss distribution with parameters $\alpha > 0$ and $\beta \in (0,1)$. Define $Y = \lfloor X \rfloor + 1$, where $\lfloor x \rfloor$ is the largest integer not more than $x$.  
  Then for every $t, x \ge 1$, $Y$ has tail function $\widebar{F}_{Y}(x) = \widebar{F}(\lfloor x \rfloor)$. As a result,
\begin{align*}
\frac{\widebar{F}_{Y}(tx)}{\widebar{F}_Y(t)} - \frac{1}{x^{\alpha}} &= \frac{\lfloor t \rfloor^{\alpha}}{\lfloor tx \rfloor^{\alpha}}\frac{1 + \lfloor tx \rfloor^{-\beta}}{1 + \lfloor t \rfloor^{-\beta}} - \frac{1}{x^{\alpha}}\\
		&= \Big(\frac{\lfloor t \rfloor^{\alpha}}{\lfloor tx \rfloor^{\alpha}} - \frac{1}{x^{\alpha}}\Big)\frac{1 + \lfloor tx \rfloor^{-\beta}}{1 + \lfloor t \rfloor^{-\beta}} + \frac{1}{x^{\alpha}}\Big(\frac{1 + \lfloor tx \rfloor^{-\beta}}{1 + \lfloor t \rfloor^{-\beta}} - 1\Big)\\
		&= \Big(\frac{\lfloor t \rfloor^{\alpha}}{\lfloor tx \rfloor^{\alpha}} - \frac{1}{x^{\alpha}}\Big)\frac{1 + \lfloor tx \rfloor^{-\beta}}{1 + \lfloor t \rfloor^{-\beta}} + \frac{1}{x^{\alpha}}\frac{\lfloor tx \rfloor^{-\beta} - \lfloor t \rfloor^{-\beta}}{1 + \lfloor t \rfloor^{-\beta}}.
\end{align*}
The second term yields
  the leading convergence order
  $t^{-\beta}$
  since
  the first
  term is of order
  $t^{-1} = o(t^{-\beta})$
  due to the bound $t-1 < \lfloor t \rfloor \le t$,
  since $\beta \in (0,1)$.
 Hence, in Definition~\ref{djklaa} we can take $A(t) = t^{-\beta}$ and $\rho = -\beta$, which proves that the discretized Hall-Weiss distribution is of second-order regular variation with parameters $-\alpha$ and $-\beta$. Therefore, as soon as $\alpha > 1$, we can apply e.g. Corollary~\ref{SRV_Wass}
 to obtain the Wasserstein bound 
\[
\mathrm{d}_{W}(Z_{n}, \mathcal{F}(\alpha)) \le \frac{C_{\gamma}}{n^{
    (1 - {1}/{\gamma}){\beta}/{\alpha}}}
\]
for every $\gamma \in (1,\alpha)$, where $Z_{n} = n^{-1/\alpha}\max(Y_{1},\dots, Y_{n})$ and the $Y_{i}$ are \iid copies of $Y$. 
\end{enumerate}
\end{example}

\subsection{Auxiliary results}
\noindent
 In this subsection, 
 we prove several facts of independent interest which 
 will be used for the proof of 
 Theorem~\ref{main_theorem}. 
\begin{proposition}\label{3_argmax}
Let $(X_{n})_{n\ge 1}$ be a sequence of \iid random variables whose tail distribution function $\widebar{F}$ belongs to $\SRV{-\alpha}{\rho}$ with auxiliary function~$A$, for some $\alpha > 0$ and $\rho < 0$, and set
\[
U_n \coloneq \big\lb \exists i \in \lbra 1,n \rbra,\ \forall j \in \lbra 1,n \rbra \bs \lb i \rb,\ X_{i} > X_j \big\rb.
\]
Then, for some constant $C > 0$ we have 
\[
\prob(U_{n}^{c}) \leq C\Big(\frac{1}{n} + A(a_n)\Big),
\quad n \geq 1, 
\]
where $a_n \coloneq F^{\leftarrow}(1 - n^{-1})$. 
\end{proposition}

\begin{proof}
 Letting $\varepsilon_n = \kappa n^{-1/(2\alpha)}$, where
 $\kappa >0$ is such that none of the $\varepsilon_{n}$ is an atom of $\widebar{F}$, we have
\[
\underset{n \to \infty}{\lim} a_n\varepsilon_n = +\infty.
\]
Using the notation $W_n \coloneq a_n^{-1}M_{n-1}$, we have 
\begin{align*}
\prob(U_{n}^{c}) &= 1 - \prob (\exists i \in \lbra 1,n \rbra,\ \forall j \neq i,\ X_{i} > X_j )\\
		&= 1 - n\prob (\forall j \in \lbra 1, n-1 \rbra,\ X_n > X_j )\\
		&= 1 - n\esp [\widebar{F}(M_{n-1})]\\
		&= 1 - n\esp [\widebar{F}(a_{n}W_{n})]\\
		&= 1 - n\esp [\widebar{F}(a_{n}W_{n})\ind_{\lb W_n > \varepsilon_{n}\rb}] - n\esp [\widebar{F}(a_{n}W_{n})\ind_{\lb W_n \leq \varepsilon_{n}\rb} ]\\
		&= 1 - n\widebar{F}(a_n)\esp [\widebar{F}_{a_{n}}(W_{n})\ind_{\lb W_n > \varepsilon_{n}\rb} ] - n\esp [\widebar{F}(a_{n}W_{n})\ind_{\lb W_n \leq \varepsilon_{n}\rb} ], 
\end{align*}
 and letting $Q_{n}\coloneq \widebar{F}_{a_{n}}(W_{n}) = \widebar{F}(a_{n}W_{n})/\widebar{F}(a_n)$ and $E_{n} = \lb W_n > \varepsilon_{n}\rb$, we find 
\begin{align*}
\prob\big(\exists i \neq j \in &\lbra 1,n \rbra,\ X_{i} = X_j \quad \mathrm{and} \quad X_{i} \ge X_{1}, \dots, X_n \big)\\ 
		&= 1 - \esp[Q_{n}\ind_{E_{n}}] - (n\widebar{F}(a_n) - 1)\esp[Q_{n}\ind_{E_{n}}] - n\esp [\widebar{F}(a_{n}W_{n})\ind_{E_{n}^{c}}]\\
&= \esp [1 - W_{n}^{-\alpha}\ind_{E_{n}}] + \esp [(W_{n}^{-\alpha} - Q_{n})\ind_{E_{n}}] - (n\widebar{F}(a_n) - 1)\esp[Q_{n}\ind_{E_{n}}] - n\esp [\widebar{F}(a_{n}W_{n})\ind_{E_{n}^{c}} ]
\\
		&\eqcolon T_{1} + T_{2} - T_{3} - T_{4}.
\end{align*}

\noindent
$(i)$ 
 Bounding $T_{1}$ follows from Lemma~\ref{lemma_T_1}.   

\noindent
$(ii)$ Regarding $T_{2}$, since $\rho < 0$, we know that
$\widebar{F}(t) = \ell L(t)/t^{\alpha}$, for some $\ell > 0$ and slowly-varying function $L$ converging to $\ell$ at infinity.
We apply the Potter bounds \eqref{2_SORV_ineq} to $t = a_n$ and $x = W_{n}$,
with $\varepsilon = 1$ and $\delta \in (0, (\alpha - \rho)/2)$.
Since $a_n\varepsilon_n $ tends to $+\infty$, we see that
 for $n$ large enough, 
 $a_{n}W_{n}$ will be eventually greater than $t_0$ on
 the event $E_n = \lb W_n \ge \varepsilon_n \rb$. Hence, we have 
\begin{align*}
\vert T_{2} \vert & \leq \esp\big[\vert W_{n}^{-\alpha} - Q_{n}\vert\ind_{E_{n}}\big]\\
		&\leq A(a_n)\esp\big[W_{n}^{-\alpha + \rho}\max\big(W_{n}^{\delta}, W_{n}^{-\delta}\big)\ind_{E_{n}} \big] -\rho^{-1}A(a_n)\esp\big[W_{n}^{-\alpha}\big\vert W_{n}^{\rho} - 1\big\vert\ind_{E_{n}} \big].
\end{align*}
 It follows from Lemma \ref{lemma_T_1} that the sequence $(\esp[W_{n}^{\gamma}\ind_{E_{n}}])_{n\ge 1}$ is bounded for every $\gamma < \alpha$, thereby proving that $T_{2}$ is $O(A(a_n))$.

\noindent
$(iii)$ Next, we deal with $T_{3}$.
A direct application of the \eqref{o(An)} shows that the term
$n\widebar{F}(a_n) - 1$ converges to $0$ faster than $A(a_n)$. Similarly to
the above, we exploit the fact that $a_n\varepsilon_{n}$ tends to infinity to apply the Potter inequalities to bound $Q_{n}$ by $2\max(W_{n}^{-\alpha/2}, W_{n}^{\alpha/2})\ind_{E_{n}}$ for $n$ large enough. We conclude from the fact that
 $\esp[\max(W_{n}^{-\alpha/2}, W_{n}^{\alpha/2})\ind_{E_{n}}]$ is the general term of a bounded sequence. 

\noindent
$(iv)$ 
The term $T_{4}$ can be simply bounded as 
\begin{align*}
n\esp\big[\widebar{F}(a_{n}W_{n})\ind_{E_{n}^{c}}\big] &\leq n\prob(W_n \leq \varepsilon_{n}) \sim ne^{-\varepsilon_{n}^{-\alpha}},\ \text{as\ } n\ {tends\ to\ } \infty.
\end{align*}
\end{proof}
\begin{lemma}\label{lemma_T_1}
 Let $\alpha > 0, \rho < 0$, and
 assume that $(\varepsilon_{n})_{n \ge 1}$ is a sequence in $(0,\infty)$
 that contains no atoms of $\widebar{F} \in \SRV{-\alpha}{\rho}$, 
 and such that 
\[
\underset{n \to \infty}{\lim} \varepsilon_n = 0 \quad \mathrm{and} \quad\underset{n \to \infty}{\lim} a_n\varepsilon_n = +\infty.
\]
Then, there exists $C>0$ such that the following bound holds:
\[
\vert\esp [W_{n}^{-\alpha}\ind_{E_{n}} ] - 1 \vert \leq C\Big(\frac{1}{n} + A(a_n)\Big)
\]
where $E_n = \lb W_n > \varepsilon_{n}\rb$. Furthermore, for every $\gamma \in (-\infty, \alpha)$, the sequence $(\esp[Z_{n}^{\gamma}\ind_{E_{n}}])_{n \ge 1}$ is bounded. 
\end{lemma}

\begin{proof}
  First, we note that $\esp[Z^{-\alpha}] = 1$ if $Z \sim \mathcal{F}(\alpha)$,
  hence we have 
\begin{align*}
1 = \esp [Z^{-\alpha}\ind_{\lb Z \leq \varepsilon_n \rb} ] + \esp [Z^{-\alpha}\ind_{\lb Z \ge \varepsilon_n \rb}] &= (1 + \varepsilon_{n}^{-\alpha})e^{-\varepsilon_{n}^{-\alpha}} + \esp[Z^{-\alpha}\ind_{\lb Z \ge \varepsilon_n \rb}]\\
&= (1 + \varepsilon_{n}^{-\alpha})e^{-\varepsilon_{n}^{-\alpha}}
+ \alpha \int_{\varepsilon_{n}}^\infty \prob(Z \leq t) \frac{\dif t}{t^{\alpha + 1}}\\
		&= (1 + \varepsilon_{n}^{-\alpha})e^{-\varepsilon_{n}^{-\alpha}} + \varepsilon_{n}^{-\alpha}\int_0^{1}\prob (Z \leq \varepsilon_n u^{-1/\alpha}) \dif u.
\end{align*}
On the other hand, we have 
\begin{align*}
  \esp [W_{n}^{-\alpha}\ind_{E_{n}} ] &= \int_0^{\varepsilon_{n}^{-\alpha}}\esp [\ind_{\lb t \leq W_{n}^{-\alpha} \leq \varepsilon_{n}^{-\alpha} \rb} ] \dif t
  \\
   &= \int_0^{\varepsilon_{n}^{-\alpha}}\prob(W_{n}^{-\alpha} \ge t) - \prob(W_{n}^{-\alpha} \ge \varepsilon_{n}^{-\alpha}) \dif t\\
		&=  -\varepsilon_{n}^{-\alpha}\prob(W_n \leq \varepsilon_{n}) + \int_0^{\varepsilon_{n}^{-\alpha}}\prob (W_n \leq u^{-1/\alpha} ) \dif u\\
		&=  -\varepsilon_{n}^{-\alpha}\prob(W_n \leq \varepsilon_{n}) + \varepsilon_{n}^{-\alpha}\int_0^{1}\prob (W_n \leq \varepsilon_{n}u^{-1/\alpha}) \dif u.
\end{align*}
 Combining the above inequalities yields 
\begin{align*}
\vert\esp [W_{n}^{-\alpha}\ind_{E_{n}}] - 1\vert
&\leq (1 + \varepsilon_{n}^{-\alpha})e^{-\varepsilon_{n}^{-\alpha}} + \varepsilon_{n}^{-\alpha}F^{n-1}(a_n\varepsilon_{n})
\\
 & \quad + \varepsilon_{n}^{-\alpha}\int_0^{1}\big\vert F^{n-1}(a_n\varepsilon_n u^{-1/\alpha}) - e^{-\varepsilon_{n}^{-\alpha}u} \big\vert\dif u, 
\end{align*}
 where the above integrand can be bounded as 
\begin{align*}\label{3_T_1}
  \vert F^{n-1}(a_n\varepsilon_n u^{-1/\alpha}) - e^{-\varepsilon_{n}^{-\alpha}u} \vert
  &\leq \vert F^{n-1}(a_n\varepsilon_n u^{-1/\alpha}) - e^{-(n-1)\widebar{F}(a_n\varepsilon_{n}u^{-1/\alpha})} \vert
  \\
  & \quad
   + \vert e^{-(n-1)\widebar{F}(a_n\varepsilon_{n}u^{-1/\alpha})} - e^{-\varepsilon_{n}^{-\alpha}u} \vert. \numberthis
\end{align*}
\noindent
$(i)$
 We start with the first half of \eqref{3_T_1}.
 For any $u \in (0, 1)$, we have 
\begin{align*}
  F^{n-1} (a_n\varepsilon_{n}u^{-1/\alpha}) = \big(1 -
   \widebar{F}(a_n\varepsilon_{n}u^{-1/\alpha})\big)^{n-1}.
\end{align*}
Since $a_n\varepsilon_{n}$ tends to infinity,
 for $n$ is large enough,
 $\widebar{F}(a_n\varepsilon_{n}u^{-1/\alpha})$ is less than $1/2$ for every $u \in [0,1]$, hence the inequality 
\begin{align*}
0 \leq e^{-y} - \Big(1 - \frac{y}{n}\Big)^{n} \leq \frac{y^{2}}{2n}e^{-y},\quad y \in \Big[0, \frac{n}{2}\Big], 
\end{align*}
 yields 
\begin{align*}
  \big\vert F^{n-1}(a_n\varepsilon_n u^{-1/\alpha}) - e^{-(n-1)\widebar{F}(a_n\varepsilon_{n}u^{-1/\alpha})} \big\vert &\leq \frac{1}{2(n-1)}\big(
   (n-1)\widebar{F}(a_n\varepsilon_n u^{-1/\alpha})\big)^{2}e^{-(n-1)\widebar{F}(a_n\varepsilon_n u^{-1/\alpha})}.
\end{align*}
Since the function $x^{2}e^{-x}$ is bounded over $\R_{+}$, the change of variable $v = \varepsilon_{n}^{-\alpha}u$ and the dominated convergence theorem imply 
\begin{align*}
  \underset{n \to \infty}{\lim}\varepsilon_{n}^{-\alpha}\int_0^{1}\big(
   (n-1)\widebar{F} (a_n\varepsilon_n u^{-1/\alpha})\big)^{2}e^{-(n-1)\widebar{F}(a_n\varepsilon_n u^{-1/\alpha})} \dif u = \int_0^\infty u^{2}e^{-u} \dif u = 2,
\end{align*}
 hence there exists $C>0$ such that for every $n \ge 1$ we have 
\[
\varepsilon_{n}^{-\alpha}\int_0^{1}\big\vert F^{n-1} (a_n\varepsilon_n u^{-1/\alpha}) - e^{-(n-1)\widebar{F}(a_n\varepsilon_{n}u^{-1/\alpha})} \big\vert \dif u \leq \frac{C}{n}\cdotp
\] 
\noindent
$(ii)$
Next, we deal with the second half of $\eqref{3_T_1}$.
 From the inequalities 
\[
\vert e^{-x} - e^{-y}\vert \leq e^{-\min(x,y)}\vert x - y\vert \leq (e^{-x} + e^{-y})\vert x - y \vert,\quad x, y \geq 0, 
\]
 and 
\[
(n-1)\widebar{F}(a_n\varepsilon_n x) \geq k(\varepsilon_{n}x)^{-\alpha},\quad x \in [1, +\infty)
\]
 which follows from \eqref{2_Bnd_tail}
 for $k \in (0,\ell ]$ 
 and $n$ sufficiently large, 
we have 
\begin{align}
  \nonumber
  \varepsilon_{n}^{-\alpha}\int_0^{1}\big\vert &e^{-(n-1)\widebar{F}(a_n\varepsilon_{n}u^{-1/\alpha})} - e^{-\varepsilon_{n}^{-\alpha}u} \big\vert \dif u\\
\nonumber
  		&\leq \varepsilon_{n}^{-\alpha}\int_0^{1} \big(e^{-(n-1)\widebar{F}(a_n\varepsilon_{n}u^{-1/\alpha})} + e^{-\varepsilon_{n}^{-\alpha}u}\big)\big\vert (n-1)\widebar{F}(a_n\varepsilon_{n}u^{-1/\alpha}) - \varepsilon_{n}^{-\alpha}u \big\vert \dif u\\
\label{ajkaa}
 &\leq 2\varepsilon_{n}^{-\alpha}\int_0^{1} e^{-k'\varepsilon_{n}^{-\alpha}u}\big\vert (n-1)\widebar{F}(a_n\varepsilon_{n}u^{-1/\alpha}) - \varepsilon_{n}^{-\alpha}u \big\vert \dif u,
\end{align}
 where $k' \coloneq \min(1,k) > 0$. 
 Next, since 
\[
(n-1)\widebar{F}(a_n) - 1 = \Big(1 - \frac{1}{n}\Big)n\widebar{F}(a_n) - 1 = -\frac{1}{n} + o(A(a_n)) = O\Big(\frac{1}{n} + A(a_n)\Big), 
\]
we have
 \begin{align*}
\big\vert (n-1)\widebar{F}(a_n\varepsilon_{n}u^{-1/\alpha}) - \varepsilon_{n}^{-\alpha}u &\big\vert \leq \big\vert (n-1)\widebar{F}(a_n\varepsilon_{n}u^{-1/\alpha}) - \widebar{F}_{a_{n}}(\varepsilon_{n}u^{-1/\alpha})\big\vert + \big\vert \widebar{F}_{a_{n}}(\varepsilon_{n}u^{-1/\alpha}) - \varepsilon_{n}^{-\alpha}u \big\vert\\
		&= \widebar{F}_{a_{n}}(\varepsilon_{n}u^{-1/\alpha})\vert (n-1)\widebar{F}(a_n) - 1 \vert + \big\vert \widebar{F}_{a_{n}}(\varepsilon_{n}u^{-1/\alpha}) - \varepsilon_{n}^{-\alpha}u \big\vert\\
		&\leq C\varepsilon_{n}^{-\alpha}u\Big(\frac{1}{n} + A(a_n)\Big) +\big\vert \widebar{F}_{a_{n}}(\varepsilon_{n}u^{-1/\alpha}) - \varepsilon_{n}^{-\alpha}u \bigg\vert. 
\end{align*}
 Using \eqref{2_SORV_ineq}, we find:
\begin{align*}
e^{-k'\varepsilon_{n}^{-\alpha} u}\big\vert \widebar{F}_{a_{n}}(\varepsilon_{n}u^{-1/\alpha}) - \varepsilon_{n}^{-\alpha}u \big\vert &\leq A(a_n)\psi_{n}(\varepsilon_{n}u^{-1/\alpha}),
\end{align*}
where the function 
\[
\psi(x) := e^{-k'x^{-\alpha}}\big(\vert \rho \vert^{-1}x^{-\alpha}(x^{\rho} + 1) + \varepsilon x^{-\alpha + \rho}\max(x^{\delta}, x^{-\delta})\big),\quad x \in [\varepsilon_{n}, +\infty) 
\]
is integrable over $(0,\infty)$. Combining this with
 \eqref{ajkaa} yields the bound
\begin{align*}
\varepsilon_{n}^{-\alpha}\int_0^{1}\big\vert e^{-(n-1)\widebar{F}(a_n\varepsilon_{n}u^{-1/\alpha})}& - e^{-\varepsilon_{n}^{-\alpha}u} \big\vert \dif u \leq C' \Big(\frac{1}{n} + A(a_n)\Big).
\end{align*}
\noindent 
Regarding the second half of the lemma, assume first that $\gamma
\in [0, \alpha)$. In
  this case, we bound $\esp[Z_{n}^{\gamma}\ind_{E_{n}}]$ by $\esp[(Z_{n})_{+}^{\gamma}]$ and apply \cite[Theorem~2.1]{Pickands68}, which yields that this last sequence converges to $\esp[Z_{+}^{\gamma}]$ when $Z \sim \mathcal{F}(\alpha)$.
   When $\gamma < 0$, the only problem concerns the small values of $Z_{n}\ind_{E_{n}}$ and we have 
\begin{align*}
\esp\big[Z_{n}^{\gamma}\ind_{E_{n}}\ind_{\lb Z_n \leq 1 \rb}\big] = \esp\big[Z_{n}^{\gamma}\ind_{\lb Z_n \leq x,\ \varepsilon_n \leq Z_n \leq 1 \rb}\big] &= \int_{\varepsilon_{n}}^\infty \gamma x^{\gamma - 1}\prob(\varepsilon_n \leq Z_n \leq 1) \dif x\\
&\leq \int_{\varepsilon_{n}}^\infty \gamma x^{\gamma - 1} F^{n}(a_{n}x) \dif x
\\
	&\leq \int_{\varepsilon_{n}}^\infty \gamma x^{\gamma - 1} e^{-n\widebar{F}(a_{n}x)} \dif x \\
	&\leq \int_{\varepsilon_{n}}^\infty \gamma x^{\gamma - 1} e^{-{x^{-\alpha/2}}} \dif x,
\end{align*} 
 for $n$ large enough, thanks to Potter's bounds
 \eqref{1_Potter}, 
 since $a_n\varepsilon_{n}$ tends to infinity as $n$ goes to infinity.
\end{proof}
\noindent 
 Next is a corollary of Proposition~\ref{3_argmax}. 
\begin{corollary}
  Let $(\bm{X}_{n})_{n \ge 1}$ be a sequence of \iid random vectors in $\R_{+}^{d}$, 
  such that the tail distribution function $\widebar{F}_{j}$ of each marginal $X_{1, j}$ 
   belongs to $\SRV{\alpha}{\rho_{j}}$ with auxiliary function~$A_{j}$, for a common $\alpha > 0$
  and possibly different $\rho_{j} < 0$,
  $j=1,\ldots , d$. Then,
  for some $C > 0$ and every $n \geq 1$ we have 
\[
\prob\big(\exists j \in \lbra 1, d \rbra,\ \exists i' \neq i \in \lbra 1,n \rbra,\ X_{i, j} = X_{i', j}\quad \mathrm{and} \quad X_{i, j} \ge X_{1, j}, \dots, X_{n, j} \big) \leq C\Big(\frac{1}{n} + A_{j_{0}}(a_{n, j_{0}})\Big),
\]
where $a_{n, j} = F_{j}^{\leftarrow}(1 - n^{-1})$ and $j_{0}$ is such that $\rho_{j_{0}} \geq \rho_{j}$ for every $j \in \lbra 1, d \rbra$.
\end{corollary}
\begin{proof}
 We apply the union bound 
\begin{align*}
\prob\big(\exists i' \neq i \in \lbra 1,n \rbra,\ \exists j \in& \lbra 1, d \rbra,\ X_{i, j} = X_{i', j}\quad \mathrm{and} \quad X_{i, j} \ge X_{1, j}, \dots, X_{n, j} \big)\\
		&= \prob\Big(\bigcup_{j \in \lbra 1, d \rbra} \big\lb i' \neq i \in \lbra 1,n \rbra,\ X_{i, j} = X_{i', j}\quad \mathrm{and} \quad X_{i, j} \ge X_{1, j}, \dots, X_{n, j} \big\rb\Big), 
\end{align*}
 and Proposition~\ref{3_argmax} to each term of the sum.
\end{proof}
\noindent 
In what follows,
we deal with the possibility
that a sequence of random variables with support larger than $(0,\infty)$ converges in law to a Fréchet distribution.
For instance, the maximum of $n$ \iid Cauchy random variables divided by $n$
converges to the Fréchet $\mathcal{F}(1)$.
However, since the infinitesimal generator of the Fréchet semi-group is
only defined on $(0,\infty)$, we must deal with the negative values of the approaching random variable before applying Stein's method. This is an instance of the so-called \textit{problem of support}, see for example
\cite[p.~23]{Dobler12}. As above,
 the next lemma is stated in the multivariate setting.
\begin{lemma}\label{lemma_pos}
Let $(\bm{X}_{i})_{i \ge 1}$ be a sequence of \iid random vectors with common distribution function $F$. Set 
\[
a_n \coloneq F^{\leftarrow}(1 - 1/n) \quad \mathrm{and} \quad \bm{M}_n \coloneq \max(\bm{X}_{1},\dots,\bm{X}_{n}),
\]
as well as $Z_n \coloneq a_n^{-1}M_{n}$. Let also $(V_{i,j})_{i \ge 1, j \in \lbra 1,d \rbra}$ be a sequence of \iid random variables, independent of the $\bm{X}_{i}$s and with distribution $\mathcal{U}_{[0.5, 1]}$. Then, if one sets $X_{i,j}' \coloneq X_{i,j} \oplus V_{i,j}$ and $\bm{M}_{n}' = \max(\bm{X}_{1}',\dots, \bm{X}_{n}')$, as well as $\bm{Z}_{n}' \coloneq a_n^{-1}\bm{M}_{n}'$, one gets:
\begin{align}
    \nonumber 
\mathrm{d}_{K}(\bm{Z}_{n}, \bm{Z}_{n}') \leq 2\sum_{j=1}^{d}F_{j}(1)^{n}, 
\end{align}
with $F_{j}$ the $j$-th marginal of $F$. Under the additional assumption that $\esp\big[\vert X_{1,j}\vert] < +\infty$ for every $j \in \lbra 1,d \rbra$, we have
\begin{align}\label{lemma_pos_Wass}
\mathrm{d}_{W}(\bm{Z}_{n}, \bm{Z}_{n}') \leq a_n^{-1}\sum_{j=1}^{d}(2 + \esp [(X_{j})_{-}])F_{j}(1)^{n-1},
\end{align}
where $x_{-} := \max(-x, 0)$.
\end{lemma}
\begin{proof}
  We start with the Kolmogorov distance
  $\mathrm{d}_{K}$. By distinguishing whether $\bm{M}_{n}$ has at least one coordinate less than $1$ or is greater than $\bm{1}$, we get
\begin{align*}
\mathrm{d}_{K}(\bm{Z}_{n}, \bm{Z}_{n}') = \sup_{\bm{x} \in \R^{d}} \vert F_{\bm{Z}_{n}}(\bm{x}) - F_{\bm{Z}_{n}'}(\bm{x})\vert &= \sup_{\bm{x} \in \R^{d}} \vert\prob(\bm{M}_{n}' \leq \bm{x}) - \prob(\bm{M}_n \leq \bm{x})\vert\\
		&\leq \prob(\bm{M}_n \ngeq \bm{1}) + \sup_{\bm{x} \in \R^{d}} \vert\prob(\bm{M}_{n}' \leq \bm{x}, \bm{M}_n \ge \bm{1}) - \prob(\bm{M}_n \leq \bm{x})\vert\\
		&\leq \sum_{j=1}^{d}F_{M_{n, j}}(1) + \sup_{\bm{x} \in \R^{d}} \vert\prob(\bm{M}_n \leq \bm{x}, \bm{M}_n \ngeq \bm{1})\vert\\
		&\leq 2\sum_{j=1}^{d}F_{M_{n, j}}(1),
\end{align*} 
the bound on $\prob(\bm{M}_n \ngeq \bm{1})$ resulting from the union bound applied to $\prob(\bm{M}_n \ngeq \bm{1})$. We have also used that if $\bm{M}_{n}'$ is greater than $\bm{1}$, then $\bm{M}_{n}' = \bm{M}_{n}$.
 As for the Wasserstein distance $\mathrm{d}_{W}$, we have 
\begin{align*}
  \mathrm{d}_{W}(\bm{Z}_{n}, \bm{Z}_{n}') \leq \sum_{j=1}^{d}\esp [\vert Z_{n,j}' - Z_{n,j}\vert ] &= a_n^{-1}\sum_{j=1}^{d}\esp\Big[
    \big(
    \underset{1 \leq i \leq n}{\max}V_{i,j} - M_{n,j}
    \big)\ind_{\lb M_{n,j} \leq \underset{1 \leq i \leq n}{\max}V_{i,j}\rb}\Big]\\
		&\leq a_n^{-1}\sum_{j=1}^{d}\esp [(1 - M_{n,j}) \ind_{\lb M_{n,j} \leq 1 \rb} ].
\end{align*}
 To conclude the proof of 
 \eqref{lemma_pos_Wass}, 
we bound $\esp[\vert M_{n,j} \vert\ind_{\lb M_{n,j} \leq 1\rb}]$
as follows:
\begin{align*}
\esp\big[\vert M_{n,j} \vert\ind_{\lb M_{n,j} \leq 1\rb}\big] &= \int_0^\infty  \P (\vert M_{n,j} \vert \ge x,\ M_{n,j} \leq 1 ) \dif x\\
		&= \int_0^\infty  \P ( M_{n,j} \ge x,\ M_{n,j} \leq 1 ) \dif x + \int_0^\infty  \P (M_{n,j} \leq -x,\ M_{n,j} \leq 1 ) \dif x\\
		&= \int_0^{1} \P ( M_{n,j} \ge x,\ M_{n,j} \leq 1 ) \dif x + \int_0^\infty  \P(M_{n,j} \leq -x) \dif x\\
		&\leq F_{j}(1)^{n} + \int_0^\infty  F_{j}(-x)^{n} \dif x\\
		&\leq F_{j}(1)^{n} + F_{j}(1)^{n-1}\int_0^\infty  F_{j}(-x) \dif x.
\end{align*}
\end{proof}

\begin{remark}
\begin{enumerate}
\item
  Thanks to the equivalence of norms in finite dimension,
   Lemma~\ref{lemma_pos}
   is not restricted to $1$-Lipschitz functions
   with respect to the $\Vert \cdot \Vert_{1}$ norm,
   namely when $\Vert \cdot \Vert_{1}$ is replaced by some arbitrary norm $\Vert \cdot \Vert$
   and $\esp[\Vert \bm{X} \Vert] < \infty$ we find 
   \[
   \mathrm{d}_{W}(\bm{Z}_{n}, \bm{Z}_{n}') \leq Ka_n^{-1}\sum_{j=1}^{d}
    (2 + \esp [(X_{j})_{-}])F_{j}(1)^{n-1},
\] 
where $K \coloneq \underset{\Vert \bm{x} \Vert = 1}{\sup}\ \Vert \bm{x} \Vert_{1}$. 

\item In dimension $d = 1$, for any $b \ge 0$ such that
   $\esp[\vert X \vert^{b}] < \infty$ we have 
  \begin{align}
      \nonumber 
0 \leq \esp [\vert Z_{n}' \vert^{b}] - \esp [\vert Z_n \vert^{b}] \leq (2 + \esp [X_{-}^{b}])F(1)^{n-1}.
\end{align}
\end{enumerate}
\end{remark}
\noindent 
 Lemma~\ref{D_a_W_n}
 is stated in the one-dimensional case.
\begin{lemma}\label{D_a_W_n}
Recall that $W_n = a_n^{-1}M_{n-1}$ and $Z_n = a_n^{-1}M_{n}$. Let $h : (0,\infty) \to \R$ be a function such that the derivative of its associated Stein solution $g_{h}'$ exists almost everywhere and satisfies
\[
x\vert g_{h}'(x)\vert \leq Cx^{\beta}
\] 
for almost every $x >0$, some $\beta \in (-\infty, \alpha)$ and $C > 0$. Then one has, for a possibly different constant $C$:
\begin{align}
    \nonumber 
\vert \esp [\Da g_{h}(Z_{n}) ] - \esp [\Da g_{h}(W_{n})] \vert \leq \frac{C}{n}\cdotp
\end{align}
\end{lemma}
\begin{proof}
Due to Lemma \ref{lemma_pos}, at the cost of a term vanishing exponentially fast with respect to $n$, we can and will assume that the $X_{i}$ are larger than $1/2$ a.s. Since $Z_n = W_{n}\ind_{\lb M_{n-1} \ge X_n \rb} + a_n^{-1}X_{n}\ind_{\lb M_{n-1} \ge X_n \rb}$, we see that
\begin{align*}
\vert \esp [\Da g_{h}(Z_{n}) ] - \esp [\Da g_{h}(W_{n}) ] \vert &= \vert \esp [ (\Da g_{h}(Z_{n}) - \Da g_{h}(W_{n}) )\ind_{\lb X_n > M_{n-1} \rb} ] \vert\\
		&= \frac{1}{\alpha}\Big\vert \esp\Big[\int_{a_n^{-1}M_{n-1}}^{a_n^{-1}X_{n}}yg_{h}'(y)\frac{\alpha}{y^{\alpha + 1}} \dif y\ind_{\lb X_n > M_{n-1} \rb} \Big] \Big\vert\\
		&\leq C\int_0^\infty  \prob (a_n^{-1}M_{n-1} \leq y \leq a_n^{-1}X_{n}) \frac{\dif y}{y^{\alpha - \beta}}\\
		&= C\int_{(2a_n)^{-1}}^\infty  F(a_{n}y)^{n-1}\widebar{F}(a_{n}y)\frac{\dif y}{y^{\alpha - \beta}},
\end{align*}
thanks to the assumption, and because $M_{n-1}$ is greater than $1/2$ a.s. The last integral is also equal to:
\begin{align*}
\int_{(2a_n)^{-1}}^\infty  F(a_{n}y)^{n-1}\widebar{F}(a_{n}y)\frac{\dif y}{y^{\alpha - \beta}} &= \widebar{F}(a_n)\int_{(2a_n)^{-1}}^\infty  F(a_{n}y)^{n-1}\widebar{F}_{a_{n}}(y)\frac{\dif y}{y^{\alpha - \beta}}.
\end{align*}
We already know that $\widebar{F}(a_n)$ is equivalent to $1/n$, so we only have to prove that the integral is bounded with respect to $n$. Let $t_0 > 0$ be such that for $n \ge t_0$ and $ny \ge t_0$, by
the Potter bound \eqref{1_Potter} applied to $\delta = \alpha$
we have 
\begin{equation}
\label{ajsa1} 
\widebar{F}_{a_{n}}(y) \leq \frac{2}{y^{\alpha}}\max (y^{\alpha}, y^{-\alpha} ) = 2\max (1, y^{-2\alpha} ).
\end{equation} 
Then, we write:
\begin{align*}
\int_{(2a_n)^{-1}}^\infty  F(a_{n}y)^{n-1}\widebar{F}(a_{n}y)\frac{\dif y}{y^{\alpha - \beta}} & = \int_{(2a_n)^{-1}}^{t_0 / a_n } F(a_{n}y)^{n-1}\widebar{F}(a_{n}y)\frac{\dif y}{y^{\alpha - \beta}}\\
&
\quad
 + \int_{t_0 a_n^{-1}}^\infty  F(a_{n}y)^{n-1}\widebar{F}(a_{n}y)\frac{\dif y}{y^{\alpha - \beta}}.
\end{align*}
The first integral is bounded as 
\begin{align*}
\int_{(2a_n)^{-1}}^{t_0 / a_{n}} F(a_{n}y)^{n-1}\widebar{F}(a_{n}y)\frac{1}{y^{\alpha - \beta}}\dif y &\leq \frac{2^{\alpha - \beta - 1}}{\alpha - \beta - 1}a_n^{\alpha - \beta - 1}F(t_0 )^{n-1}, 
\end{align*}
 which tends to zero as $n$ tends to infinity,
 and the second integral can be bounded through Potter's bounds
 \eqref{ajsa1} as 
\begin{align*}
  \int_{t_0 / a_{n} }^\infty  F(a_{n}y)^{n-1}\widebar{F}_{a_{n}}(y)\frac{\dif y}{y^{\alpha - \beta}} &\leq 2\int_{t_0 / a_{n} }^\infty  F(a_{n}y)^{n-1}\max
  (1, y^{-2\alpha}
   )\frac{\dif y}{y^{\alpha - \beta}}.
\end{align*}
 We note
 that the integrand evaluated at $t_0 / a_{n} < 1$ satisfies
\[
\lim_{n \to \infty}
2F(t_0 )^{n-1}(t_0^{-1}a_n)^{3\alpha -\beta} = 0 
\]
 hence for $n$ large enough,
 the integrand is less than $1$ in a
  neighbourhood of $t_0 / a_{n}$, which proves that this sequence of integrals is bounded as $n$ tends to infinity. 
\end{proof}
\noindent 
An example of a set of test functions $h$ satisfying the assumptions of Lemma \ref{D_a_W_n} is the set of $1$-Lipschitz functions.
The next proposition shows that an even stronger property is satisfied by $g_{h}'$ for such $h$,
 therefore improving the estimates given in \cite{Costaceque24_mlti}. 
\begin{lemma}\label{Lemma_x_dif_g_h}
Let $h : \R \to \R$ be a $1$-Lipschitz function and $\alpha$ a positive scalar. Then, for every $x \in \Rp$ we have 
\begin{align}
  xg_{h}'(x) &= -\alpha\int_0^{x}e^{{x^{-\alpha}}-{u^{-\alpha}}}h'(u)\dif u \label{x_dif_g_h_alt}
  \\
&= -\alpha (h(x) - h(0)) + \alpha^2 \int_0^{x} (h(u) - h(0))
e^{{x^{-\alpha}}- {u^{-\alpha}}  }\frac{\dif u}{u^{\alpha + 1}} .\label{x_dif_g_h_Hol}
\end{align} 
Consequently, the function $x \mapsto xg_{h}'(x)$ is $2\alpha$-Lipschitz on $(0,\infty)$.
 Furthermore, if $\alpha < 1$ and $h \in \mathcal{H}_b$ for some $b \in (0,\alpha)$, then the following bounds holds:
\begin{align}\label{x_dif_g_h_Hol_bnd}
x\vert g'_{h}(x) \vert \leq 2\alpha x^b,\quad x \in \Rp. 
\end{align}
\end{lemma}
\begin{proof}
 Relation~\eqref{x_dif_g_h_alt}
 is consequence of
 \eqref{Stein_sol_diff} and the change of variable $u = xe^{-t/\alpha}$, 
 Next, letting $x, y > 0$ such that $x < y$ 
 without loss of generality,
  \eqref{x_dif_g_h_alt} yields 
\begin{align*}
  xg_{h}'(x) - yg_{h}'(y) = \alpha\int_{x}^{y}
  e^{{y^{-\alpha}}-{u^{-\alpha}} }h'(u) \dif u
  + \alpha\int_0^{x}
  \big(
  e^{{y^{-\alpha}}-{u^{-\alpha}} }
  -
  e^{{x^{-\alpha}}-{u^{-\alpha}} }\big) h'(u) \dif u, 
\end{align*}
 hence 
\begin{align*}
  \frac{1}{\alpha}\vert xg_{h}'(x) - yg_{h}'(y) \vert &\leq \int_{x}^{y}
  e^{{y^{-\alpha}}- {u^{-\alpha}} } \dif u + \int_0^{x}\Big\vert
  e^{{y^{-\alpha}}-{u^{-\alpha}} } -
  e^{{x^{-\alpha}}-{u^{-\alpha}} }\Big\vert \dif u\\
		&\leq \vert x - y \vert + \big[e^{{x^{-\alpha}}} - e^{{y^{-\alpha}}}\big]\int_0^{x}e^{-{u^{-\alpha}}} \dif u\\
		&= \vert x - y \vert + \big[e^{{x^{-\alpha}}} - e^{{y^{-\alpha}}}\big]\int_0^{x}\frac{u^{\alpha + 1}}{\alpha} \frac{\alpha}{u^{\alpha + 1}} e^{-{u^{-\alpha}}} \dif u\\
		&\leq \vert x - y \vert + \frac{x^{\alpha + 1}}{\alpha}\big[e^{{x^{-\alpha}}} - e^{{y^{-\alpha}}}\big]e^{-{x^{-\alpha}}}\\
		&\leq 2\vert x - y\vert, 
\end{align*}
 where we used the fact that $z \mapsto \exp(z^{-\alpha})$ is $\alpha x^{-(\alpha + 1)}\exp(x^{-\alpha})$-Lipschitz over $[x, y]$.
Finally, to prove \eqref{x_dif_g_h_Hol_bnd}, we use \eqref{x_dif_g_h_Hol} and bound the integral term by 
\begin{align*}
  \alpha
  \Big\vert\int_0^{x}(h(u) - h(0))
  e^{{x^{-\alpha}}-{u^{-\alpha}} }\frac{\dif u}{u^{\alpha + 1}} \Big\vert
  &\le
  \alpha
  \int_0^{x}u^b
  e^{{x^{-\alpha}}-{u^{-\alpha}} }\frac{\dif u}{u^{\alpha + 1}} \\
  &\le
  \alpha
  x^b \int_0^{x}
  e^{{x^{-\alpha}}-{u^{-\alpha}} }\frac{\dif u}{u^{\alpha + 1}} \\
		&= x^b.
\end{align*}

\end{proof}

\begin{remark}
\begin{itemize}

\item[-] The Lipschitz property of $x \mapsto xg'_{h}(x)$ can
  also be proved
  by noting that since
  $g_{h}$ solves the Stein equation $\La g_{h} = h(x) - \esp[h(Z)]$, where $Z \sim \mathcal{F}(\alpha)$, it satisfies
\[
\frac{1}{\alpha} xg_{h}'(x) = \Da g_{h}(x) - h(x) + \esp[h(Z)], 
\]
 and that since $h$ is 1-Lipschitz
and $\vert g_{h}'(x) \vert \leq \min(\alpha, x^{\alpha})$,
 for every $0 < x < y$ we have 
\[
\big\vert \Da g_{h}(x) - \Da g_{h}(y) \big\vert \leq \int_{x}^{y}
\min (\alpha, r^{\alpha}) \frac{\dif r}{r^{\alpha}} \leq \vert x - y\vert.
\]
\item[-] We observe the similarity between
   \eqref{x_dif_g_h_alt} and the identity
\[
xg_{z}'(x) = \alpha e^{{x^{-\alpha}}-{z^{-\alpha}} }\ind_{(z, +\infty)}(x)
\]
deduced from \eqref{g_z_diff}.

\item[-] Lastly, note that
  since $u \mapsto \alpha u^{-(\alpha + 1)}
  e^{-(u^{-\alpha} - x^{-\alpha})}
  \ind_{[0,x]}(u)$ is the density of the distribution of
  $Z \sim \mathcal{F}(\alpha)$ 
  conditional on being less than $x$,
  \eqref{x_dif_g_h_Hol} can be further rewritten as
\[
xg_{h}'(x) = \alpha\esp [h(x) - h(Z) \/ Z \leq x ]. 
\]

\end{itemize}
\end{remark}
\noindent 
The set of power functions $x^{b}$ with $b \in (0, \alpha)$
provides another example of test functions such that 
 \eqref{H_1}
 and
 \eqref{H_2}
 are satisfied.
\begin{lemma}\label{Lemma_g_b}
  Let $b \in (0, \alpha)$ and $h_{b}(x) = x^{b}$. The derivative of
  the Stein solution $g_{b}$ associated with $h_{b}$
   is infinitely differentiable on $(0,\infty)$ and satisfies:
\begin{align}\label{g_b_diff}
g'_{b}(x) = -\int_0^\infty \frac{bx^{b-1}}{(u+1)^{1+{b}/{\alpha}}}e^{-{u}{x^{-\alpha}}}\dif u,\quad x \in \Rp, 
\end{align}
 and we have 
\begin{align}\label{g_b_diff_ineq}
  x\vert g'_{b}(x) \vert \leq bx^{b}\min
  ( \alpha / b , x^{\alpha}
  ),\quad x \in \Rp.  
\end{align}
Moreover, the second derivative of $g_{b}$ satisfies: 
\begin{align}
\label{g_b_diff2_ineq}
    x\vert g''_{b}(x) \vert &\leq bx^{b-1}
    \big(
    \vert b-1 \vert\min(\alpha / b , x^{\alpha}) + \alpha \min(x^{-\alpha}, x^{\alpha})
    \big),\quad x \in \Rp.
\end{align}
\end{lemma}
\begin{proof}
  Relation~\eqref{g_b_diff}
  follows from \eqref{Stein_sol_diff}
  and the change of variable $u = x^{-\alpha}(e^{t} - 1)$. 
 To prove \eqref{g_b_diff_ineq},
 we note that
 the left-hand side of \eqref{g_b_diff}
 is clearly bounded by $b x^b \alpha/b$,
 while the right-hand side is dominated by $b x^b x^{\alpha}$
 since the denominator of the integrand is greater than~$1$.
By differentiating the right-hand side of \eqref{g_b_diff} with respect to $x$, we find
\begin{align*}
g_{b}''(x) &= (b-1)\frac{1}{x}g_{b}'(x) - \frac{\alpha b}{x^{\alpha - b + 2}}\int_0^\infty  \frac{u}{(u+1)^{1 + {b}/{\alpha}}}e^{-{u}{x^{-\alpha}}} \dif u\\
		&= (b-1)\frac{1}{x}g_{b}'(x) - \alpha bx^{\alpha + b-2}\int_0^\infty  \frac{u}{(x^{\alpha}u+1)^{1 + {b}/{\alpha}}}e^{-u} \dif u.
\end{align*}
 Hence, \eqref{g_b_diff2_ineq} follows from
 \eqref{g_b_diff_ineq} and the inequalities 
\[
 \int_0^\infty  \frac{u}{(u+1)^{1 + {b}/{\alpha}}}e^{-{u}{x^{-\alpha}}} \dif u \leq
 \int_0^\infty  e^{-{u}{x^{-\alpha}}} \dif u = x^{\alpha}, 
\] 
\[
\int_0^\infty  \frac{u}{(x^{\alpha}u+1)^{1 + {b}/{\alpha} }}e^{-u} \dif u
\leq
 \int_0^\infty  ue^{-u} \dif u = 1.
\]
\end{proof}

\subsection{Proof of the main results} 

We prove the main results of this section: Theorem~\ref{main_theorem} and Theorem~\ref{SRV_Kol}.

\subsubsection{Proof of Theorem~\ref{main_theorem}}

The idea of the proof consists in starting from one half of the expectation of the generator, $\esp[Z_{n}g_{h}'(Z_{n})]$, and rewrites it as something more easily comparable to the second half $\alpha\esp[\Da g_{h}(Z_{n})]$. We do so by conditioning on each of the events $\lb X_{i} > M_{n\bs i}\rb$, that is, $X_{i}$ strictly exceeds the other observations and use the regular variation assumption to exploit the fact that $M_{n\bs i}$ tends to infinity when $n$ and thus will play the role of a high threshold. In turn, this explains why a Pareto random variable is featured in the expression of $\Da$ (see \eqref{Da_alt_2}). This method is inspired by the leave-one-out approach initially developed by Stein in \cite{Stein72} for sums of random variables.

\begin{proof}
  Due to Lemma \ref{lemma_pos}, we can and will assume that the $X_{i}$ are positive, and even take values in $[1/2, +\infty)$.
    The resulting discrepancy term will be shown to decrease
    exponentially fast to $0$, \textit{i.e.} much faster than
    the announced rate of convergence,
    and therefore will be neglected.
 If the random variables $X_{i}$ are continuous, it holds that
\begin{align*}
  \esp[h(Z_{n})] - \esp[h(Z)] = -\esp [\La g_{h}(Z_{n})]
  = -\esp [\La g_{h}(Z_{n})\ind_{U_{n}}],
\end{align*}
where the event 
\[
U_n \coloneq \big\lb \exists i \in \lbra 1,n \rbra,\ \forall i' \in \lbra 1,n \rbra \bs \lb i \rb,\ X_{i} > X_{i'} \big\rb
\]
 has been defined in Proposition \ref{3_argmax}, and $\La g_{h}(x) = -\alpha^{-1}xg_{h}'(x) + \Da g_{h}(x)$. 
If not, we can write 
\begin{align*}
  \esp[h(Z_{n})] - \esp[h(Z)] &= \esp [(h(Z_{n}) - \esp[h(Z)])\ind_{U_{n}}] + \esp [ (h(Z_{n}) - \esp[h(Z)])\ind_{U^{c}_{n}}]
  \\
		&= \esp [\La g_{h}(Z_{n})\ind_{U_{n}}] + \esp [(h(Z_{n}) - \esp[h(Z)])\ind_{U^{c}_{n}}].
\end{align*}
Letting
$\psi (x) \coloneq xg_{h}'(x)$
and $M_n \coloneq \max(X_{1},\dots,X_{n})$,
and denoting by $X$ an independent copy of $X_{n}$,
$n\geq 1$, we have 
\begin{align*}
  \esp [\psi(Z_{n})\ind_{U_{n}}]
  & = \sum_{i=1}^{n}\esp [\psi(a_n^{-1}X_{i})\ind_{\lb X_{i} > M_{n\bs i} \rb} ]
  \\
  &= n\esp [\psi(a_n^{-1}X)\ind_{\lb X > M_{n-1} \rb} ]
\\
 &= n\esp\big[\widebar{F}(M_{n-1})\esp [\psi(a_n^{-1}X) \/ X > M_{n-1}, M_{n-1} ]\big]\\
&= n\esp\big[\widebar{F} (a_n W_{n})\esp [\psi(W_{n}X /
    M_{n-1})\ |\ X > M_{n-1}, M_{n-1} ]\big], 
\end{align*}	
where 
 $W_n \coloneq a_n^{-1}M_{n-1}$.
 We will compare this term to
 $\alpha\esp[\Da g_{h}(W_{n})\ind_{U_{n}}]$ 
 and use Lemma \ref{D_a_W_n} to 
 bound the difference with $\alpha\esp[\Da g_{h}(Z_{n})\ind_{U_{n}}]$ as 
 \[
 \vert\esp [\Da g_{h}(W_{n})\ind_{U_{n}}] - \esp [\Da g_{h}(Z_{n})\ind_{U_{n}}] \vert \leq \frac{C}{n}
\]
for some constant $C>0$ independent of $h$.
We have
\begin{align*}
\vert\esp [&\psi(Z_{n})\ind_{U_{n}}] - \alpha\esp [\Da g(W_{n})\ind_{U_{n}} ] \vert\\
		&= \big\vert \esp \big[n\widebar{F}(a_{n}W_{n})\esp [\psi(W_{n}M_{n-1}^{-1}X)\ |\ X > M_{n-1}, M_{n-1} ]\ind_{U_{n}}\big] - \esp\big[W_{n}^{-\alpha}\esp [\psi(W_{n}Y) \/ M_{n-1} ]\ind_{U_{n}} \big]\big\vert\\
		&\leq B_1(n) + B_2(n),
\end{align*}
where 
$Y \sim \mathcal{VP}(\alpha)$ is a Pareto random variable with tail distribution function $y^{-\alpha}$ on $[1, +\infty)$,
  \begin{align}
    \nonumber 
B_1(n) \coloneq \esp\big[n\widebar{F}(a_{n}W_{n})\big\vert \esp [ \vert\psi(W_{n}M_{n-1}^{-1}X) \vert \/ X > M_{n-1}, M_{n-1} ] - \esp [\psi(W_{n}Y) \/ M_{n-1}]\big\vert\big],
\end{align}
and
  \begin{align}
      \nonumber 
B_2(n) \coloneq \esp\big[\big\vert n\widebar{F}(a_{n}W_{n}) - W_{n}^{-\alpha} \big\vert\esp [ \vert\psi(W_{n}Y) \vert \/ M_{n-1}]\big].
\end{align}
 We bound $B_2(n)$ using the inequality 
\begin{align*}
\big\vert n\widebar{F}(a_{n}x) - \frac{1}{x^{\alpha}}\big\vert &\leq \big\vert n\widebar{F}(a_{n}x) - \widebar{F}_{a_{n}}(x)\big\vert + \big\vert \widebar{F}_{a_{n}}(x) - \frac{1}{x^{\alpha}}\big\vert\\
		&\leq \widebar{F}_{a_{n}}(x)\vert n\widebar{F}(a_n) - 1\vert + A(a_n) x^{-\alpha}\Big\vert\frac{x^{\rho} - 1}{\rho}\Big\vert + A(a_n)x^{-\alpha + \rho}\max (x^{\rho/2}, x^{-\rho/2})
\end{align*} 
that follows from \eqref{2_SORV_ineq}
under the condition $a_{n}x > t_0$.
 The first term above is 
$\vert n\widebar{F}(a_n) - 1\vert = o(A(a_n))$
by Lemma~\ref{djkla1a}.
 Regarding the second term, using \eqref{H_2}, we have 
\begin{align*}
  \esp\big[W_{n}^{-\alpha}\vert W_{n}^{\rho} - 1\vert \esp [ \vert\psi(W_{n}Y) \vert \/ M_{n-1} ] \big] &\leq C_{1}\esp\big[W_{n}^{\beta - \alpha}\vert W_{n}^{\rho} - 1\vert \esp [Y^{\beta} \/ M_{n-1} ]\big]
  \\
		&= C_{1}\esp [W_{n}^{\beta - \alpha}\vert W_{n}^{\rho} - 1\vert ] \esp[Y^{\beta}]\\
		&= \frac{\alpha}{\alpha - \beta} C_{1}\esp [W_{n}^{\beta - \alpha}\vert W_{n}^{\rho} - 1\vert], 
\end{align*}
 which is bounded in $n\geq 1$ from Lemma \ref{lemma_T_1}. Thus, $B_2(n)$ is of order $A(a_n)$.

 \medskip

 \noindent
  As for $B_1(n)$, using the identities 
\[
\esp [\psi(wY)] = \psi(w) + w\int_{1}^\infty \psi'(wy) y^{-\alpha}\dif y
\]
and, since $\prob(m^{-1}X \ge y \/ X \ge m) = \widebar{F}_{m}(y)$ whenever $y$ is greater than $1$:
\[
\esp [\psi(wm^{-1}X) \/ X > m ] = \psi(w) + w\int_{1}^\infty \psi'(wy)\widebar{F}_{m}(y) \dif y, 
\]
 as well as \eqref{H_2}, 
we write 
\begin{align*}\label{dist_to_Pareto}
\vert\esp [\psi(wm^{-1}X) \/ X > m ] - \esp [\psi(wY) ] \vert &\leq w\int_{1}^\infty  \vert \psi'(wy) \vert \big\vert \widebar{F}_{m}(y) - y^{-\alpha} \big\vert \dif y\\
		&\leq C_{2}w^{\gamma + 1}\int_{1}^\infty  y^{\gamma}\big\vert \widebar{F}_{m}(y) - y^{-\alpha} \big\vert \dif y\\
		&\leq C_{2}w^{\gamma + 1}A(m)\int_{1}^\infty y^{\gamma - \alpha}\Big(\frac{y^{\rho} - 1}{\vert\rho\vert} + y^{\rho/2} \Big) \dif y\\
&\eqcolon C_{\alpha, \rho}w^{\gamma+1}A(m),
\numberthis
\end{align*}
 for every $w >0$ and $m > t_0$,
 where $t_0$ is given in \eqref{2_SORV_ineq},
with $\varepsilon = 1$ and $\delta = -\rho/2$. 
We will take $m = M_{n-1}$ and $w = W_{n}$, but first we must distinguish whether $M_{n-1}$ is greater or less than $t_0$. On
the event $\lb M_{n-1} \leq t_0 \rb$, using \eqref{H_1} we have 
\begin{align*}
\esp\big[\big\vert n\widebar{F}(a_{n}W_{n}) - W_{n}^{-\alpha}\big\vert\esp [\vert\psi(W_{n}Y)\vert \/ M_{n-1}]&\ind_{\lb M_{n-1} \leq t_0 \rb}\big]\\
		&\leq C_{1}\esp [Y^{\beta} ]\esp\big[W_{n}^{\beta}\big\vert n\widebar{F}(a_{n}W_{n}) - W_{n}^{-\alpha}\big\vert\ind_{\lb M_{n-1} \leq t_0 \rb}\big]\\
&\leq \frac{1}{\alpha - \beta}C_{1}\alpha \max \big(1, a_n^{-\beta}t_0^{\beta}\big)
\esp [(n + W_{n}^{-\alpha})\ind_{\lb M_{n-1} \leq t_0 \rb} ]
\\
&\leq \frac{1}{\alpha - \beta}C_{1}\alpha
\max \big(1, a_n^{-\beta}t_0^{\beta}\big)
(n + a_n^{\alpha})F^{n-1}(t_0 )\\
		&= o(A(a_n)),
\end{align*}
because we can always take $t_0 > 1$.
Consequently, from \eqref{dist_to_Pareto} 
 we obtain 
\begin{align*}
  B_1(n) &\leq C_{\alpha, \rho}\esp\big[n\widebar{F}(a_{n}W_{n})W_{n}^{\gamma + 1}A (M_{n-1})\ind_{\lb M_{n-1} > t_0 \rb}\big] + o(A(a_n))
  \\
		&= C_{\alpha, \rho}\esp\big[n\widebar{F}(a_{n}W_{n})W_{n}^{\gamma + 1}A (a_{n}W_{n})\ind_{\lb M_{n-1} > t_0 \rb}\big] + o(A(a_n))\\
  &= C_{\alpha, \rho}A (a_n)\esp\Big[n\widebar{F}(a_{n}W_{n})W_{n}^{\gamma + 1}\frac{A (a_{n}W_{n})}{A (a_n)}\ind_{\lb M_{n-1} > t_0 \rb}\Big]
  + o(A (a_n)).
\end{align*}
Since $A $ is $\rho$-regularly varying, and at the cost of making $t_0$ and $C_{\alpha, \rho}$ possibly bigger, one can apply
Potter's bounds \cite[p.~23]{deHaan07} to $A $ and $\widebar{F}$,
 which yields 
\begin{align*}
B_1(n) &\leq C_{\alpha, \rho}A (a_n)\esp\big[n\widebar{F}(a_{n}W_{n})W_{n}^{\gamma + 1 + \rho}\max(W_{n}^{\rho/2}, W_{n}^{-\rho/2})\ind_{\lb M_{n-1} > t_0 \rb}\big] + o(A (a_n))\\
		&\leq C_{\alpha, \rho}A (a_n)\esp\big[W_{n}^{\gamma + 1 - \alpha + \rho}\max (W_{n}^{(\rho - \alpha)/2}, W_{n}^{-(\rho + \alpha)/2})\big] + o(A (a_n)).
\end{align*}
We conclude from Lemma \ref{lemma_T_1}, which
shows that the above sequence of expectations is bounded in $n\geq 1$.
\end{proof}

\subsubsection{Proof of Theorem~\ref{SRV_Kol}}

Next is the proof of Theorem \ref{SRV_Kol}. It follows exactly the same lines as the previous proof, except we have to circumvent the fact that Assumption \eqref{H_2} is not satisfied when the test function $h = \ind_{(-\infty, z]}$ is the indicator function of an half-line.

\begin{proof}
 Here, $h_{z} = \ind_{(-\infty, z)}$ does not satisfy \eqref{H_2}, as 
the function $xg_{z}'(x) = \alpha e^{{x^{-\alpha}}-{z^{-\alpha}}} \ind_{(z, \infty]}(x)$ 
is not absolutely continuous. However, \eqref{H_1} is still satisfied
from $xg_{z}'(x) \leq \alpha$ by taking
$\beta = 0$ and $C_{1} = \alpha$,
hence only the proof of inequality \eqref{dist_to_Pareto}
needs modification.
For this,
 taking $Z \sim \mathcal{F}(\alpha)$
 and applying the defining property of the Stein solution $g_{z}$, \textit{i.e.} 
\[
\frac{1}{\alpha}\psi(x) = \frac{1}{\alpha}xg_{z}'(x) = \Da g_{z}(x) - \ind_{[0, z]}(x) + \prob(Z \leq z),
\]
 to $x = wm^{-1}X$ and $y = wY$ and taking expectations, we find
\begin{align*}
\frac{1}{\alpha} \big\vert\esp [\psi(wm^{-1}X) \/ X > &m ] - \esp [\psi(wY)] \big\vert\\
		&\leq \big\vert \widebar{F}_{m}(w^{-1}) - \prob(Y \ge w^{-1})\big\vert + \big\vert\esp [\Da g_{z}(wm^{-1}X)\/ X > m ] - \esp [\Da g_{z}(wY) ]\big\vert\\
		&= \big\vert \widebar{F}_{m}(1 \oplus w^{-1}) - (1 \oplus w^{-1})^{-\alpha} \big\vert + \int_0^\infty g_{z}'(y) \big\vert\widebar{F}_{m}(w^{-1}y) - \prob(Y \ge w^{-1}y)\big\vert \frac{\dif y}{y^{\alpha}}\\
		&\leq \big\vert \widebar{F}_{m}(1 \oplus w^{-1}) - (1 \oplus w^{-1})^{-\alpha}\big\vert + \alpha\int_{w}^\infty  \big\vert\widebar{F}_{m}(w^{-1}y) - \prob(Y \ge w^{-1}y)\big\vert \frac{ \dif y}{y^{\alpha + 1}}\\
		&= \big\vert \widebar{F}_{m}(1 \oplus w^{-1}) - (1 \oplus w^{-1})^{-\alpha} \big\vert + \alpha w^{-\alpha}\int_{1}^\infty  \big\vert\widebar{F}_{m}(y) - \prob(Y \ge y)\big\vert \frac{ \dif y}{y^{\alpha + 1}}.
\end{align*}
As in the proof of \eqref{dist_to_Pareto}, we assume that $m > t_0$ and $w > 0$, where $t_0$ is given by \eqref{2_SORV_ineq} applied to $\varepsilon = 1$ and $\delta = -\rho/2$. Since everything can be bounded by a non-increasing function of $y$, we can replace $1 \oplus w^{-1}$ by $1$ to obtain 
\begin{align*}
\frac{1}{\alpha}\big\vert\esp [\psi(wm^{-1}X) \/ X > m ] - \esp [\psi(wY) ]\big\vert \leq A (m)(2\vert\rho\vert^{-1} + 1) + \alpha w^{-\alpha}A (m)\int_{1}^\infty  \frac{1 - y^{\rho}}{\vert \rho \vert} \frac{\dif y}{y^{2\alpha + 1}}.
\end{align*}
The rest of the proof remains unchanged. 
\end{proof}

\section{Bounds using density functions}
\label{s4}
\noindent
\subsection{Main results}
\noindent 
Theorem~\ref{Cor_Int_all} is a consequence of Lemma~\ref{bnd_Frcht_abs}
below, specialized to the cases where $h$ is equal to a $1$-Lipschitz function, an indicator function $h_{z} = \ind_{(-\infty, z]}$, a power function $h_{b} = x^{b}$ and a $b$--H{\"o}lder function.
 \begin{theorem} 
   \label{Cor_Int_all}
   Let
  $Z \sim \mathcal{F}(\alpha)$ for some $\alpha > 0$,
  and let 
  $W$ be a random variable whose distribution function $F_W$
  is absolutely continuous on $[K, +\infty)$ for some $K>0$, with 
  density function~$f_W$.
\begin{enumerate}[i)] 
\item Kolmogorov bound. For any $\alpha >0$, we have
\begin{align*}
      \nonumber 
\mathrm{d}_{K} (W, \mathcal{F}(\alpha) ) \leq (\alpha + 1)F_{W}(K) + \alpha\int_{K}^\infty  F_{W}(r)\bigg\vert\frac{f_{W}(r)}{F_{W}(r)} - \frac{\alpha}{r^{\alpha + 1}}\bigg\vert \dif r.\numberthis
\end{align*}
\item Wasserstein bound.
  Let $\alpha > 1$ and
  $p \in (1,\alpha)$ such that $\esp[\vert W \vert^{p}] < \infty$.
  We have
  \begin{align*}
      \nonumber 
      \mathrm{d}_{W} (W, \mathcal{F}(\alpha) ) \leq \esp [\vert W \vert^{p} ]^{{1}/{p}}F_{W}(K)^{{1}/{q}} & + \Big(
      \Gamma\Big(1 - \frac{1}{\alpha}\Big) + \frac{\alpha^{2}}{\alpha - 1}\Big)
      F_{W}(K)\\
		&+ \alpha\int_{K}^\infty  rF_{W}(r)\bigg\vert\frac{f_{W}(r)}{F_{W}(r)} - \frac{\alpha}{r^{\alpha + 1}}\bigg\vert \dif r,\numberthis
\end{align*}
  where $q = p/(p-1)$ denotes the conjugate of $p$. 
\item Moment bound.
  Let $\alpha > 0$ and
  $b \in (0, \alpha)$,
  assume that $W$ is a.s. nonnegative and that there exists $p \in (1, \alpha / b)$ such that $\esp[W^{pb}] < \infty$, with $q = p/(p-1)$. Then, we have  
  \begin{align*}
      \nonumber 
\vert \esp [W^{b} ] - \esp [Z^{b}]\vert \leq \esp [ W^{pb}]^{{1}/{p}}F_{W}(K)^{{1}/{q}} & + \Big( \Gamma\Big(1 - \frac{b}{\alpha}\Big) + \frac{\alpha^{2}}{\alpha - b}\Big) F_{W}(K)\\
		&+ \alpha\int_{K}^\infty  r^{b}F_{W}(r)\bigg\vert\frac{f_{W}(r)}{F_{W}(r)} - \frac{\alpha}{r^{\alpha + 1}}\bigg\vert \dif r.\numberthis
\end{align*}

\item Let $\alpha \in (0,1)$ and $b \in (0, \alpha)$.
  Suppose that there exists $p \in (1, \alpha / b)$ such that $\esp[\vert W \vert^{p}]< \infty$, with $q = p/(p-1)$. Then, we have 
  \begin{align*}
      \nonumber 
      \mathrm{d}_{W_b} (W, \mathcal{F}(\alpha) )
      \leq \esp [ W^{pb}]^{{1}/{p}}F_{W}(K)^{{1}/{q}}
      &
      + \Big(
      \Gamma\Big(1 - \frac{b}{\alpha}\Big) + \frac{2\alpha^{2}}{\alpha - b}\Big)
       F_{W}(K)\\
		&+ \alpha\int_{K}^\infty  r^{b}F_{W}(r)\bigg\vert\frac{f_{W}(r)}{F_{W}(r)} - \frac{\alpha}{r^{\alpha + 1}}\bigg\vert \dif r.\numberthis
\end{align*}
\end{enumerate}
 \end{theorem}
 \noindent 
 As above, we let
 $(X_{i})_{i \geq 1}$ denote an \iid sequence of regularly varying
 random variables with index $-\alpha < 0$ and common
 distribution function $F$,
 and let 
\[
M_{n} = \max (X_1,\ldots , X_n) \quad \text{and} \quad Z_{n} = a_n^{-1}M_{n},\ n\geq 1 
\]
where $(a_{n})_{n\geq 1}$ is a sequence of positive numbers such that
 $\lim_{n\to \infty} n\widebar{F}(a_{n}) = 1$. 
Notice that the regular variation assumption implies that for every
$b \in (0, \alpha)$ and $p \in (1, \alpha / b)$ we have 
\[
\esp [\vert X_{1} \vert^{pb} ] < +\infty, 
\] 
hence $\esp[\vert W \vert^{pb}]$ is indeed finite. 

\medskip
 
 In order to apply Theorem~\ref{Cor_Int_all}
 to the derivation of
 convergence rates in the extreme value theorem,
 we assume that there exists $\eta > 0$
 such that $F$ is absolutely continuous on $[\eta, +\infty)$, with density $f$,
 and take 
 $W \coloneq Z_{n}$,
 $K \coloneq \eta / a_{n}$, $n\geq 1$.
 Notice that 
\[
F_{W}(K) = F(\eta)^{n}
\] 
and so vanishes exponentially fast as $n$ goes to infinity. Furthermore, thanks to \cite{Pickands68}, the sequence of expectations $(\esp[Z_{n}^{\beta}])_{n \ge 1}$ converges to $\esp[Z^{\beta}]$ where $Z \sim \mathcal{F}(\alpha)$, for any $\beta \in [0, \alpha)$. As a result, this sequence is bounded. Consequently, we only need to focus on the integral term of each of the four inequalities of Theorem~\ref{Cor_Int_all}, and therefore we want to bound for every $\beta \in [0, \alpha)$ the integral
    \begin{align}
\nonumber 
\int_{K}^{\infty} r^{\beta}F_{W}(r)\Bigg\vert\frac{f_{W}(r)}{F_{W}(r)} - \frac{\alpha}{r^{\alpha + 1}}\Bigg\vert \dif r = \int_{a_n^{-1}\eta}^{\infty} r^{\beta}F(a_{n}r)^{n}\Bigg\vert na_n\frac{f(a_{n}r)}{F(a_{n}r)} - \frac{\alpha}{r^{\alpha + 1}}\Bigg\vert \dif r,
\end{align}
    where we have used that $f_{W}(r) = na_{n}F(a_{n}r)^{n-1}f(a_{n}r)$. Intuitively, the term $F(a_{n}r)^{n}$ increasingly cancels the weight of small values of $r$ as $n$ goes to infinity, and tends to $1$ for very large $r$,
     while the difference in absolute value tends to explode for small $r$, but vanishes for large arguments. 
 The change of variable $s = a_{n}r$ yields:
\begin{align*}\label{log_diff_gamma}
  \int_{a_n^{-1}\eta}^{\infty} r^{\beta}F(a_{n}r)^{n}\Bigg\vert na_n\frac{f(a_{n}r)}{F(a_{n}r)} - \frac{\alpha}{r^{\alpha + 1}}\Bigg\vert \dif r &= \frac{1}{a_{n}}
  \int_\eta^{\infty} (a_n^{-1}s)^{\beta}F(s)^{n}\Bigg\vert na_n\frac{f(r)}{F(s)} - a_n^{\alpha + 1}\frac{\alpha}{s^{\alpha + 1}}\Bigg\vert \dif s\\
		&= a_n^{-\beta}\int_\eta^{\infty} s^{\beta}F(s)^{n}\Bigg\vert n\frac{f(s)}{F(s)} - a_n^{\alpha}\frac{\alpha}{s^{\alpha + 1}}\Bigg\vert \dif s\\
		&\le na_n^{-\beta}\int_\eta^{\infty} s^{\beta}F(s)^{n}\Bigg\vert \frac{f(s)}{F(s)} - \frac{\alpha}{s^{\alpha + 1}}\Bigg\vert \dif s, \numberthis
\end{align*}
 which controls the convergence rates obtained when 
 applying Theorem~\ref{Cor_Int_all} to 
 $W = Z_{n}$ and $K = \eta / a_{n}$, $n\geq 1$.
\noindent 
Lemma~\ref{Cor_pract_bnd} provides sufficient conditions
for the derivation of explicit
convergence rates for \eqref{log_diff_gamma}.
\begin{lemma}
  \label{Cor_pract_bnd}
  Let $(X_{i})_{i \geq 1}$ be a sequence of regularly varying \iid random variables with index $-\alpha < 0$ and common
 distribution function $F$.
 Assume that $\widebar{F}(x) \sim x^{-\alpha}$ as $x$ goes to infinity for some $\alpha > 0$,
 suppose that there exists $\eta >0$ such that $F$ is absolutely continuous on $[\eta, +\infty)$
   with density $f$, and
   let $\beta \in (0,\alpha)$.
   \begin{enumerate}[i)]
   \item
     If there exists
     $\eta ,
     c_F,
     C''>0$ and $\gamma > \beta + 1$ such that
\begin{equation}
  \label{4.6}
  \widebar{F}(x) \ge \frac{c_F}{x^{\alpha}},\quad x \ge \eta, 
\end{equation} 
 and 
\begin{align}
  \label{Hyp_log_diff}
\bigg\vert \frac{f(x)}{F(x)} - \frac{\alpha}{x^{\alpha + 1}}\bigg\vert \leq \frac{C''}{x^{\gamma}}, \quad x > 0, 
\end{align}
then, we have 
\begin{align}
\nonumber 
\int_\eta^\infty  r^{\beta}F(r)^{n}\bigg\vert\frac{f(r)}{F(r)} - \frac{\alpha}{r^{\alpha + 1}}\bigg\vert \dif r
\le
\frac{C''}{\alpha}
c_{F}^{-({\gamma - \beta - 1})/{\alpha}}
  \Gamma\Big(\frac{\gamma - \beta - 1}{\alpha}\Big)\frac{1}{n^{({\gamma - \beta - 1})/{\alpha}}},
\end{align}
where $\Gamma$ denotes the Gamma function.
\item 
  If there exists $\eta , 
  c_F,C_f, \gamma > 0$ such that
\[
\widebar{F}(x) \leq \frac{C_{F}}{x^{\alpha}},\quad x \ge \eta, 
\] 
and
\begin{align}\label{Hyp_diff_bnd}
\vert x^{\alpha + 1}f(x) - \alpha \vert \leq \frac{C_f}{x^{\gamma}},\quad x \geq \eta,  
\end{align}
 then we have 
\begin{align*}\label{Cor_diff_rate_ineq}
  \int_\eta^\infty r^{\beta}F(r)^{n} & \bigg\vert \frac{f(r)}{F(r)} - \frac{\alpha}{r^{\alpha + 1}}\bigg\vert \dif r
\\
& \leq \frac{C_f}{\alpha}c_F^{
  -1 + {\beta - \gamma }/{\alpha}}\Gamma\Big(1 +\frac{\gamma - \beta}{\alpha}\Big)\frac{1}{n^{1 + ({\gamma - \beta})/{\alpha}}} + C_{F}c_F^{-2 + {\beta}/{\alpha}}\Gamma\Big(2 - \frac{\beta}{\alpha}\Big)\frac{1}{n^{2 - {\beta}/{\alpha}}}. 
\end{align*}
\end{enumerate} 
\end{lemma} 
\begin{proof}
  \noindent
  $(i)$
  Using \eqref{Hyp_log_diff} and
 the inequality $\log F(r) \leq -\widebar{F}(r)$, we have 
\begin{align}
  \int_\eta^\infty  s^{\beta}F(s)^{n}
\bigg\vert \frac{f(s)}{F(s)} - \frac{\alpha}{s^{\alpha + 1}}\bigg\vert \dif s 
&\leq C'' 
 \int_\eta^\infty  \frac{F(s)^{n} }{s^{\gamma - \beta}}\dif s
 \\
   \nonumber
&\leq C'' \int_\eta^\infty 
 \frac{e^{-n{c_F} / {s^{\alpha}}}}{s^{\gamma - \beta}}\dif s
 \\
   \nonumber
   &= C'' \int_0^{\eta^{-\alpha}} \frac{e^{-nc_{F}s}}{s^{({\beta + 1 - \gamma})/{\alpha}}}\frac{\dif s}{s}
   \\
   \nonumber
&= \frac{C'' }{\alpha}(c_{F}n)^{{-(\gamma - \beta - 1)}/{\alpha}}
   \int_0^{nc / \eta^{\alpha}} \frac{e^{-t}}{t^{({\beta + 1 - \gamma})/{\alpha}}}\frac{\dif t}{t}
   \\
   \nonumber
&\leq \frac{C'' }{\alpha}c_{F}^{-({\gamma - \beta - 1})/{\alpha}}\Gamma\Big(\frac{\gamma - \beta - 1}{\alpha}\Big)\frac{1}{n^{({\gamma - \beta - 1})/{\alpha}}}. 
\end{align}

\noindent
  $(ii)$
    We have
   \begin{align*}
 \int_\eta^\infty & r^{\beta}F(r)^{n}\bigg\vert \frac{f(r)}{F(r)} - \frac{\alpha}{r^{\alpha + 1}}\bigg\vert \dif r\\
 &\le
 \int_\eta^\infty  r^{\beta}F(r)^{n-1}\big\vert r^{\alpha + 1}f(r) - \alpha\big\vert \frac{\dif r}{r^{\alpha + 1}} + \alpha \int_\eta^\infty  r^{\beta}F(r)^{n-1}\widebar{F}(r) \frac{\dif r}{r^{\alpha + 1}}
 \\
 &\le
 C_f\int_\eta^\infty  e^{-(n-1){c_F}{r^{-\alpha}}}\frac{\dif r}{r^{\alpha + \gamma - \beta + 1}} + \alpha \int_\eta^\infty  e^{-(n-1){c_F}{r^{-\alpha}}}\widebar{F}(r) \frac{\dif r}{r^{\alpha - \beta + 1}}
 \\
 &\le
  C_f\int_\eta^\infty  r^{\beta}e^{-(n-1){c_F}{r^{-\alpha}}}\frac{\dif r}{r^{\alpha + \gamma - \beta + 1}} + C_{F}\alpha \int_\eta^\infty  e^{-(n-1){c_F}{r^{-\alpha}}} \frac{\dif r}{r^{2\alpha - \beta + 1}}
  \\
  & = 
\frac{C_f}{\alpha}c_F^{
  -1 + {\beta - \gamma }/{\alpha}}\Gamma\Big(1 +\frac{\gamma - \beta}{\alpha}\Big)\frac{1}{n^{1 + ({\gamma - \beta})/{\alpha}}} + C_{F}c_F^{-2 + {\beta}/{\alpha}}\Gamma\Big(2 - \frac{\beta}{\alpha}\Big)\frac{1}{n^{2 - {\beta}/{\alpha}}}. 
\end{align*}
\end{proof}
\noindent
Next, we consider several examples of application of
Theorem~\ref{Cor_Int_all}. 
\begin{example}
\begin{enumerate}

\item Hall--Weiss distribution.
  Just as in Example \ref{Ex_Hall-Weiss}, consider the Hall-Weiss distribution supported
  on $[1, +\infty)$. 
  In order to get a convergence to a unit-scale Fréchet distribution, we renormalize this
  distribution by $2^{-1/\alpha}$, and thus work with 
  \[
  \widebar{F} : t \mapsto \frac{1}{t^{\alpha}}\Big(1 + \frac{2^{\beta/\alpha}}{t^{\beta}}\Big),\ t \ge 2^{1/\alpha}.
  \]
  The corresponding logarithmic derivative exists over $[2^{1/\alpha}, \infty)$ and is equal to
\[
\frac{f(x)}{F(x)} = \frac{\dif}{\dif x} \log F(x) = \frac{1}{F(x)}\frac{\alpha}{x^{\alpha + 1}} + \frac{1}{F(x)}\frac{2^{\beta/\alpha}(\alpha + \beta)}{x^{\alpha + \beta + 1}}
\]
We have
\[
\frac{f(x)}{F(x)} 
 = \frac{1}{F(x)}\frac{\alpha}{x^{\alpha + 1}} + \frac{1}{F(x)}\frac{2^{\beta/\alpha}(\alpha + \beta)}{x^{\alpha + \beta + 1}}, \quad x\geq 1. 
\]
Taking $\eta = 2^{1/\alpha} + 1$ to have some leeway, we can bound $1/F(x)$ by $1/F(\eta) < \infty$ over $[2^{1/\alpha}, +\infty)$, and
  \eqref{4.6} reads 
\begin{align*}
\bigg\vert \frac{f(x)}{F(x)}  - \frac{\alpha}{x^{\alpha + 1}}\bigg\vert &= \frac{\widebar{F}(x)}{F(x)}\frac{\alpha}{x^{\alpha + 1}} + \frac{1}{F(x)}\frac{2^{\beta/\alpha}(\alpha + \beta)}{x^{\alpha + \beta + 1}}\\
		&\leq \frac{1}{F(\eta)}\Bigg(\frac{\alpha}{x^{2\alpha + 1}} + \frac{2^{\beta/\alpha}(\alpha + \beta)}{x^{\alpha + \beta + 1}} \Bigg)\\
&\leq 
\frac{
  \alpha + 2^{\beta/\alpha}(\alpha + \beta)
}{F(\eta) x^{\alpha + \min(\alpha, \beta) + 1}},
\quad
 x \ge \eta = 2^{1/\alpha} + 1. 
\end{align*}
 By application of 
 Lemma~\ref{Cor_pract_bnd}-$(i)$
 with $\gamma : = \alpha + \min(\alpha, \beta) + 1$,
 combined with inequality \eqref{log_diff_gamma},
 we obtain a rate of convergence of order $n^{-1}$
 if $\alpha \leq  \beta$,
 and $n^{-\beta/\alpha}$ if $\alpha > \beta$ for the four types of functionals we have
 been considering in this section.
 This is the same rate as in Example \ref{Ex_Hall-Weiss}. 

\item {Resnick's example \cite{Resnick87}}. In the previous example, the rate of convergence was always slower than $n^{-1}$,
  which is common among standard distributions,
  as seen in the examples.
  However,
  it is possible to find rates of convergence much faster than $n^{-1}$,
  as in the following example given in \cite{Resnick87}.
  Let 
\[
G(x) \coloneq \exp\Big(-\exp\Big(-\int_{1}^{x}(1 + e^{-u}) \frac{\dif u}{u}\Big)\Big) = \exp\Big(-\frac{1}{x}
 e^{-I(x)} \Big),\quad x \ge 1
\]
where $I(x) \coloneq
 \int_{1}^x  u^{-1} e^{-u} \dif u$. 
By a different approach, a rate of convergence of order $\exp(-c\sqrt{n})$
in the Kolmogorov distance for a certain unspecified $c > 0$
was obtained in \cite{Resnick87}. 
Letting $F (x) \coloneq G( \sigma x)$ with $\sigma = e^{-I}$,
 we have 
\[
\frac{f(x)}{F(x)} = \sigma \frac{1 + e^{-\sigma x}}{\sigma x}\exp\Big(-\int_{1}^{\sigma x}(1 + e^{-u}) \frac{\dif u}{u}\Big) = e^{I}\frac{1 + e^{-\sigma x}}{x^{2}}e^{-I(\sigma x)}. 
\]
Since $I(x)$ is positive for every $x \ge \eta = 1$, we have thanks to the inequality $1 - e^{-x} \le x$:
\[
0 \le \frac{1}{x^{2}}e^{I-I(\sigma x)} - \frac{1}{x^{2}} \le \frac{1}{x^{2}}e^{I - I(\sigma x)}\big(1 - e^{I - I(\sigma x)}\big) \le \frac{1}{\sigma x^{3}}e^{I - I(\sigma x)}e^{-\sigma x}
\] 
because 
\[
0 \le I - I(\sigma x) \le \frac{1}{\sigma x}e^{-\sigma x}.
\]
For $x$ larger than $\sigma^{-1}$, so that $\sigma x$ is greater than $1$, we get 
\[
e^{I - I(\sigma x)} \le e^{\frac{1}{\sigma x}} \le e.
\]
Therefore we find that \eqref{4.6} is satisfied with 
\[
\bigg\vert \frac{f(x)}{F(x)}
- \frac{1}{x^{2}} \bigg\vert
\leq \frac{e^{1 - \sigma x}}{x^{3}} + \frac{e^{-\sigma x}}{x^{2}}
\leq (1 + e) \frac{e^{-\sigma x}}{x^{2}}, 
\]
 hence, from the bound
$F(x)^{n} \leq e^{-n / x}$, $x \ge 1$,
 we obtain 
\begin{align*}
  \int_{1}^\infty  r^{\beta}F(r)^{n}\bigg\vert\frac{f(r)}{F(r)} - \frac{1}{r^{2}}\bigg\vert \dif r &\le
  (1 + e) \int_1^\infty r^{\beta}
  e^{\sigma r-{n}/{r}}\frac{\dif r}{r^{2}}. 
\end{align*}
Using for instance Laplace's method (see \cite{Erdelyi56}), we see that this integral is of order $n^{(2\beta - 3)/4}e^{-2\sqrt{n}}$
 as $n$ goes to infinity, hence a rate of convergence of order 
\[
n^{1 - \beta}n^{(2\beta - 3)/4}e^{-2\sqrt{\sigma n}} = n^{(1 - 2\beta)/4}e^{-2\sqrt{\sigma n}}.
\]

\item {Log-gamma distribution}. In this case,
   we have 
\[
\widebar{F}(x) = \frac{1 + \log x}{x^{\alpha}},\quad x \ge 1, 
\]
 with 
\[
 \frac{f(x)}{F(x)} = \frac{1}{x^\alpha - 1 - \log x}
\Big(\alpha\frac{1 + \log x}{x} - \frac{1}{x}\Big).
\]
Tails like these appear for instance when considering the product of two \iid Pareto random variables. An easy computation shows that if one takes $a_n = (\alpha^{-1}n^{-1}\log n)^{1/\alpha}$ as the renormalizing sequence, then $a_n^{-1}M_{n}$ converges in distribution to a unit-scale Fréchet $\mathcal{F}(\alpha)$ random variable, where $M_{n}$ is the maximum of $n$ \iid random variables with tail distribution function $\widebar{F}$. 
 Additional computations yield
\begin{align}
  \label{djkl1a} 
\bigg\vert na_n\frac{f(a_{n}x)}{F(a_{n}x)} - \frac{\alpha}{x^{\alpha + 1}}\bigg\vert &\leq \frac{\alpha}{\log n}\frac{\vert C + \alpha\log x + \log\log n \vert}{x^{\alpha + 1}}, 
\end{align} 
where $C > 0$ is independent of $n \geq 1$ and $x >0$. On the other hand, 
using the bound
\[
F(a_{n}x)^{n} \leq e^{-{x^{-\alpha}}}, \quad x > 0, 
\]
 and taking $\eta = 2$, we see that the integral
 of \eqref{djkl1a} is bounded (up to a negligible term) by 
\begin{align*}
  \frac{\alpha}{\log n}\int_{{2} / {a_{n}}}^\infty  e^{-{x^{-\alpha}}}
  \frac{
  \vert\log x \vert + \log \log n 
  }{x^{\alpha + 1}} \dif x &= \frac{1}{\log n}\int_0^{({a_{n}}/{2})^{\alpha}} e^{-u}
   (\alpha^{-1}\vert\log u \vert + \log \log n ) \dif u\\
		&\leq D\frac{\log \log n}{\log n},
\end{align*}
with $D > 0$ independent of $n \geq 1$,
and where we used the finiteness of $\int_0^\infty \vert \log u \vert e^{-u} \dif u$.
\end{enumerate}

\end{example}

\begin{remark}
Notice that in our setting, $(Z_{n})_{n\geq 1}$ cannot converge in distribution to $Z$ faster than at exponential speed, except in trivial cases. This is due to a result given in \cite{Rootzen84} which states that a faster than exponential rate of convergence in Kolmogorov distance implies that $Z_{n}$ actually is a max-stable distribution for $n$ large enough. For that reason, we can ignore safely the terms vanishing exponentially fast, like $F_{Z_{n}}(K)^{n} = F(\eta)^{n}$.
\end{remark}
\noindent 

\noindent 
 Next, we give examples of application of Lemma~\ref{Cor_pract_bnd}-$(ii)$.
\begin{example}

\begin{enumerate}


\item {Log-logistic distribution}. In this case, we have
  $\widebar{F}(x) = (1+x^{\alpha})^{-1}$ for $x\geq 0$,
  and letting $\eta = 1$ yields 
$$ 
\widebar{F}(x) = \frac{1}{x^{\alpha}}\frac{1}{1 + x^{-\alpha}} \geq \frac{1}{2x^{\alpha}},\quad x\geq 1, 
$$ 
so that we can take $c_F = 1/2$.
Next, since the function $x \mapsto (1+x)^{2}$ is $2$-Lipschitz on $[0,1]$,
 we have 
\begin{align*}
  \vert x^{\alpha+1}f(x) - \alpha \vert & =
  \alpha
  \bigg\vert \frac{x^{2\alpha}}{(1 + x^{\alpha})^{2}} - 1 \bigg\vert
  \\
   &= \alpha \Big(1 - \frac{1}{(1 + x^{-\alpha})^{2}}\Big)\\
&\leq \alpha ((1 + x^{-\alpha})^{2} - 1)
\\
&\leq  \frac{2\alpha}{x^{\alpha}},
 \quad x\geq 1.
\end{align*}
Therefore, \eqref{Hyp_diff_bnd} is satisfied with $\gamma = \alpha$, and
 from Theorem~\ref{Cor_Int_all},  
\eqref{Cor_diff_rate_ineq} yields a rate of convergence of $n^{-1}$ in
Wasserstein distance (if $\alpha > 1)$, Kolmogorov distance,
 as well as in $\mathcal{H}_{b}$ metric and for the $b$-th moment if $b \in (0,\alpha)$. 
 
\item {Student distribution}. Although the distribution function of the Student distribution with $\nu$ degrees of freedom is not available in a closed form, its density function is given by 
$$ 
f(x) = \frac{\Gamma\big(\frac{\nu + 1}{2}\big)}{\sqrt{\pi\nu}\Gamma\big(\frac{\nu}{2}\big)}\Big(1 + \frac{x^{2}}{\nu}\Big)^{-({\nu + 1})/{2}} \eqcolon \frac{1}{C}\Big(1 + \frac{x^{2}}{\nu}\Big)^{-({\nu + 1})/{2}}\cdotp 
$$
This function is regularly varying of order $\nu + 1$, so $\alpha = \nu$. Replace $f$ by $x \mapsto \sigma f(\sigma x)$ with $\sigma$ a positive scalar to be determined. We will abuse notations and keep denoting that new function by $f$:
$$ 
f(x) = \frac{\sigma}{C}\Big(1 + \frac{\sigma^{2}x^{2}}{\nu}\Big)^{-({\nu + 1})/{2}}. 
$$
We want to find $\sigma$ such that $x^{\nu+1}f(x) - \nu$ vanishes when $x$ goes to infinity, and to determine at which speed it does. For every $x \ge \eta = 1$, we have
\begin{align*}
\vert x^{\nu+1}f(x) - \nu \vert = \bigg\vert x^{\nu+1}\frac{\sigma}{C}\Big(1 + \frac{\sigma^{2}x^{2}}{\nu}\Big)^{-({\nu + 1})/{2}} - \nu\bigg\vert &= \Big\vert \frac{\sigma}{C}(\nu^{-1}\sigma^{2} + x^{-2})^{-({\nu + 1})/{2}} - \nu\Big\vert\\
		&= \bigg\vert \frac{\nu^{({\nu + 1})/{2}}}{C\sigma^{\nu}}(1 + \nu\sigma^{-2}x^{-2})^{-({\nu + 1})/{2}} - \nu\bigg\vert.
\end{align*}
We see that $\sigma$ must be equal to $C^{-1/\nu}\nu^{(\nu - 1)/2\nu}$ for the last term to vanish at infinity. Taking this value of $\sigma$, we can factorize to obtain
\begin{align*}
\vert x^{\nu+1}f(x) - \nu \vert = \nu\bigg\vert \frac{\nu^{({\nu - 1})/{2}}}{C\sigma^{\nu}}(1 + \nu\sigma^{-2}x^{-2})^{-({\nu + 1})/{2}} - 1\bigg\vert &= \nu\vert (1 + \nu\sigma^{-2}x^{-2})^{-({\nu + 1})/{2}} - 1\vert\\
		&\leq \sigma^{-2}\nu^{2}\frac{\nu + 1}{2x^{2}},
\end{align*}
as the function $x \mapsto (1 + x)^{-(\nu + 1)/2}$ is $\sigma^{-2}\nu(\nu + 1)/2$-Lipschitz on $[0,1]$. Consequently,
\eqref{Hyp_diff_bnd} is satisfied with $C_f = 2$,
 which yields a convergence rate of order $n^{-1} + n^{-2/\nu}$. An application of Smith's aforementioned result (\cite[Theorem 1]{Smith82}) confirms this is the correct rate in Kolmogorov distance, even when $\nu \in (0,2)$.

\item {Generalized Beta prime distribution}. This distribution has four positive parameters $\alpha, \beta, p, q$, the first three being shape parameters while $q$ is a scale parameter. Its density function is
\[
f(x) = \frac{p}{q\mathrm{B}(\alpha, \beta)}\frac{(x/q)^{\alpha p - 1}}{(1 + (x/q)^{p})^{\alpha + \beta}} = \frac{q^{p\beta}}{\beta\mathrm{B}(\alpha, \beta)}\frac{p\beta}{x^{p\beta + 1}}\frac{1}{(1 + (x/q)^{-p})^{\alpha + \beta}},\ x>0 
\]
where $\mathrm{B}(\alpha, \beta) = \int_0^{1}t^{\alpha - 1}(1 - t)^{\beta - 1} \dif t$ is the beta function. By choosing $q = \mathrm{B}(\alpha, \beta)^{-1/(p\beta)}$, we obtain for $x \ge K = 1$
\begin{align*}
\vert x^{p\beta + 1}f(x) - p\beta \vert &= p\beta \Big\vert\frac{1}{(1 + (x/q)^{-p})^{\alpha + \beta}} - 1\Big\vert\\
&= \frac{p\beta}{(1 + (x/q)^{-p})^{\alpha + \beta}}\Big(
\Big(1 + \Big(\frac{q}{x}\Big)^{p}\Big)^{\alpha + \beta} - 1\Big)
\\
		&\leq \frac{p\beta}{(1 + (x/q)^{-p})^{\alpha + \beta}}\\
		&\leq pq^{p}\frac{\alpha + \beta}{x^{p}}\cdotp 
\end{align*}
By the same arguments as in the previous examples, we get a rate of convergence of order $n^{-1} + n^{-p/\beta}$ for the four types of functionals we have been considering. Notice that the generalized Beta prime distribution admits several interesting particular cases, including the beta prime distribution, the Fisher distribution, the Pareto distribution, the log-logistic distribution, as well as more exotic laws (Daggum, Singh-Maddala, \textit{etc.}).

\end{enumerate}
\end{example}
\subsection{Proof of Theorem~\ref{Cor_Int_all}} 
\noindent
Lemma~\ref{bnd_Frcht_abs}
provides a general bound on the distance between an arbitrary random variable $W$ and a Fréchet $\mathcal{F}(\alpha)$ distributed random variable $Z$. 
\begin{lemma}
  \label{bnd_Frcht_abs}
  Let
  $Z \sim \mathcal{F}(\alpha)$ for some $\alpha > 0$,
  and let 
  $W$ be a random variable whose distribution function $F_W$
  is absolutely continuous on $[K, +\infty)$ for some $K>0$, with 
    density function~$f_W$.
 Let also $\mathcal{H}$ be a set of functions $h : \R \to \R$ such that $\esp[\vert h(W) \vert] < \infty$ for every $h$, the Stein solution $g_{h}$ exists, is absolutely continuous and satisfies \eqref{H_1} 
 for almost every $x > 0$,
 for some $\beta \in (-\infty, \alpha)$ and $C_{1} > 0$ independent of $h$.
 Then, we have 
\begin{align*}\label{main_theorem_ineq_intensity}
\underset{h \in \mathcal{H}}{\sup} \ \vert \esp[h(W)] - \esp[h(Z)] \vert \leq \underset{h \in \mathcal{H}}{\sup}\ \vert &\esp[h(W)\ind_{\lb W \leq K \rb}] - \esp[h(Z)]F_{W}(K)\vert\\
		&  + \frac{\alpha}{\alpha - \beta}C_{1}F_{W}(K) + C_{1}\int_{K}^\infty  r^{\beta}F_{W}(r)\bigg\vert\frac{f_{W}(r)}{F_{W}(r)} - \frac{\alpha}{r^{\alpha + 1}}\bigg\vert \dif r.\numberthis
\end{align*}
\end{lemma} 
\begin{proof}
 Letting $\bar{h} \coloneq h - \esp[h(Z)]$, we have
\begin{align*}
    \nonumber 
    \vert\esp[h(W)] - \esp[h(Z)]\vert = \vert\esp[\bar{h}(W)] \vert \leq \vert\esp[\bar{h}(W)\ind_{\lb W \leq K \rb}]\vert + \vert\esp[\bar{h}(W)\ind_{\lb W > K \rb}]\vert,
    \numberthis
\end{align*}
and the term with the indicator function $\ind_{\lb W \leq K \rb}$ yields the first half of \eqref{main_theorem_ineq_intensity}. Next,
 we rewrite the second term using the Stein solution $g_{h}$ as 
\begin{align*}
-\esp [\bar{h}(W)\ind_{\lb W > K \rb}] = \esp [\La g_{h}(W)\ind_{\lb W > K \rb} ] = -\frac{1}{\alpha}\esp [Wg'_{h}(W)\ind_{\lb W > K \rb} ] + \esp [\Da g_{h}(W)\ind_{\lb W > K \rb}].
\end{align*}
Denoting by $\rho_{\alpha}$ the measure supported on $(0,\infty )$
and defined by $\rho_{\alpha} ( [x, +\infty) ) \coloneq x^{-\alpha}$,
  $x > 0$,
  we have
\begin{align*}
\alpha\esp [\bar{h}(W)\ind_{\lb W > K \rb} ] &= \int_{K}^\infty rg'_{h}(r)f_{W}(r)\dif r - \int_{K}^\infty  rg'_{h}(r)\prob(K < W \leq r)\dif\rho_{\alpha}(r)\\
		&= \int_{K}^\infty rg'_{h}(r)f_{W}(r)\dif r - \int_{K}^\infty  rg'_{h}(r)F_{W}(r)\dif\rho_{\alpha}(r) + F_{W}(K)\int_{K}^\infty  rg'_{h}(r)\dif\rho_{\alpha}(r)\\
		&= \int_{K}^\infty  rg'_{h}(r)F_{W}(r)\Bigg(\frac{f_{W}(r)}{F_{W}(r)} \dif r - \dif\rho_{\alpha}(r)\Bigg) + F_{W}(K)\int_{K}^\infty  rg'_{h}(r)\dif\rho_{\alpha}(r).
\end{align*}
Using \eqref{H_1}, 
 this yields 
\begin{align*}
\alpha \vert\esp [\bar{h}(W)\ind_{\lb W > K \rb}]\vert &\leq C_{1}\int_{K}^\infty  r^{\beta}F_{W}(r)\bigg\vert\frac{f_{W}(r)}{F_{W}(r)} - \frac{\alpha}{r^{\alpha + 1}}\bigg\vert \dif r + C_{1}F_{W}(K)\int_{K}^\infty  r^{\beta}\frac{\alpha}{r^{\alpha + 1}}\dif r\\
		&= C_{1}\int_{K}^\infty  r^{\beta}F_{W}(r)\bigg\vert\frac{f_{W}(r)}{F_{W}(r)} - \frac{\alpha}{r^{\alpha + 1}}\bigg\vert \dif r + \frac{\alpha}{\alpha - \beta}C_{1}F_{W}(K).
\end{align*}
\end{proof}

The proof of Theorem~\ref{Cor_Int_all} goes as follows: we isolate the negligible mass that $W$ puts below a threshold $K$ and rewrites the remainder, via the Stein equation, as a single integral comparing the $f_{W}/F_{W}$ to the logarithmic derivative of the Fréchet distribution function. This difference is then bounded pointwise using the density estimates~\eqref{H_1}–\eqref{H_2}.

\begin{Proofy} {\em of Theorem~\ref{Cor_Int_all}}.
  The integral part of each bound stems from Lemma~\ref{bnd_Frcht_abs} and
  the properties of the Stein solution $g_{h}$ when $h$ is
  respectively 1-Lipschitz, equal to $x \mapsto x^{b}$, to $\ind_{(-\infty, z]}$,
  or belongs to $\mathcal{H}_{b}$ for some $b \in (0, \alpha)$ and $z > 0$,
  with each case corresponding
  respectively to $\beta = 1, b, 0$ and $b$.
  We also have $C_{1} = \alpha$ in the first three cases,
  and $C_{1} = 2\alpha$ in the last one,
  see Lemmas~\ref{Lemma_x_dif_g_h} and \ref{Lemma_g_b}.
To conclude, we bound the term $\vert \esp[h(W)\ind_{\lb W \leq K \rb}] - \esp[h(Z)]F_{W}(K)\vert$ uniformly in $h$. When $h$ is $1$-Lipschitz and $\alpha > 1$, we have 
\begin{align*}
  \vert \esp[h(W)\ind_{\lb W \leq K \rb}] - \esp[h(Z)]F_{W}(K)
   \vert &\leq \esp[Z]F_{W}(K) + \esp[\vert W \vert\ind_{\lb W \leq K \rb}]\\
		&= \Gamma\Big(1 - \frac{1}{\alpha}\Big)F_{W}(K) + \esp[\vert W \vert\ind_{\lb W \leq K \rb}].
\end{align*}
The assumption of the statement gives $p \in (1, \alpha)$ such that $\esp[\vert W \vert^{p}]< \infty$. By Hölder's inequality, we have 
\[
\esp[\vert W \vert\ind_{\lb W \leq K \rb}] \leq \esp [\vert W \vert^{p} ]^{{1}/{p}}F_{W}(K)^{{1}/{q}}.
\]
If $h = x^{b}$ for some $b \in (0, \alpha)$, we have similarly
\begin{align*}
  \vert \esp [W^{b}\ind_{\lb W \leq K \rb} ] - \esp [Z^{b}]F_{W}(K) \vert &\leq \esp [Z^{b} ]F_{W}(K) + \esp [W^{b}\ind_{\lb W \leq K \rb}]
  \\
		&\leq \Gamma\Big(1 - \frac{b}{\alpha}\Big)F_{W}(K) + \esp [\vert W \vert^{pb} ]^{{1}/{p}}F_{W}(K)^{{1}/{q}}.
\end{align*} 
When $h = \ind_{(-\infty, z]}$, $h$ takes values in $\lb 0,1 \rb$ and so
\begin{align*}
\vert \esp[h(W)\ind_{\lb W \leq K \rb}] - \esp[h(Z)]F_{W}(K)\vert \leq F_{W}(K).
\end{align*}
Finally, when $h \in \mathcal{H}_{b}$, Hölder's inequality once again yields
\begin{align*}
\vert \esp[h(W)\ind_{\lb W \leq K \rb}] - \esp[h(Z)]F_{W}(K)\vert &\leq \esp [\vert Z\vert^{b} ]F_{W}(K) + \esp [\vert W \vert^{b}\ind_{\lb W \leq K \rb} ]\\
		&\leq \Gamma\Big(1 - \frac{b}{\alpha}\Big)F_{W}(K) + \esp [\vert W \vert^{pb} ]^{{1}/{p}}F_{W}(K)^{{1}/{q}}.
\end{align*} 
\end{Proofy}
\noindent 

\footnotesize

\bibliography{biblio} 

@misc{Anastasiou21,
author = {Anastasiou, A. and Barp, A. and Briol, F. and Ebner, B. and Gaunt, R. and Ghaderinezhad, F. and Gorham, J. and Gretton, A. and Ley, C. and Liu, Q. and Mackey, L. and Oates, C. and Reinert, G. and Swan, Y.},
title = {Stein's Method Meets Statistics: A Review of Some Recent Developments},
  note   = {Preprint arXiv:2105.03481}, 
  url          = {https://arxiv.org/abs/2105.03481},
  year          = {2021}
}

@misc{Bartholome13,
  author = {C. Bartholome AND Y. Swan},
  title  = {Rates of convergence towards the {F}réchet distribution},
  note   = {Preprint arXiv:1311.3896}, 
  url          = {https://arxiv.org/abs/1311.3896},
  year   = {2013}
}

@article{Chen24,
  title={Multivariate stable approximation by {S}tein's Method},
  author={Chen, P. and Nourdin, I. and Xu, L. and Yang, X.},
  journal={Journal of Theoretical Probability},
  volume={37},
  number={1},
  pages={446--488},
  year={2024},
  publisher={Springer}
}

@article{Cheng01,
  title={The {E}dgeworth expansion for distributions of extreme values},
  author={Cheng, S. and Jiang, C.},
  journal={Science in China Series A: Mathematics},
  volume={44},
  number={4},
  pages={427--437},
  year={2001},
  publisher={Springer}
}

@article{Coutin24,
  title={New approaches to {CLT} for stable random variables},
  author={Coutin, L. AND Decreusefond, L. AND Huang, L.},
  journal={To be released},
  year={2024},
}

@article{Cohen82,
 ISSN = {00018678},
 URL = {http://www.jstor.org/stable/1427026},
 author = {J. P. Cohen},
 journal = {Advances in Applied Probability},
 number = {4},
 pages = {833--854},
 publisher = {Applied Probability Trust},
 title = {Convergence Rates for the Ultimate and Penultimate Approximations in Extreme-Value Theory},
 volume = {14},
 year = {1982}
}

@article{Costaceque24,
  title={Stein's method for max-stable random vectors},
  author={Costac{\`e}que, B. and Decreusefond, L.},
  journal={arXiv preprint arXiv:2507.00463},
  year={2025},
}

@article{Costaceque24_cpn,
  title={Convergence rate for the coupon collector’s problem with {S}tein's method},
  author={Costac{\`e}que, B. and Decreusefond, L.},
  journal={Stochastic Processes and their Applications},
  pages={104835},
  year={2025},
  publisher={Elsevier}
}

@article{Costaceque24_mlti,
  title={Functional analysis of multivariate max-stable distributions},
  author={Costac{\`e}que, B. and Decreusefond, L.},
  journal={arXiv preprint arXiv:2509.02200},
  year={2025},
}

@Article{Decreusefond15,
  author  = {Decreusefond, L.},
  journal = {ESAIM: Proceedings},
  title   = {The {{Stein}}-{{Dirichlet}}-{{Malliavin}} Method},
  year    = {2015},
  doi     = {10.1051/proc/201551003},
  pages   = {11},
  url     = {https://partage.imt.fr/index.php/s/FGmZiZMLXz3FJfN},
}

@Article{deHaan96,
 ISSN = {00911798},
 URL = {http://www.jstor.org/stable/2244834},
 author = {L. de Haan and S. Resnick},
 journal = {The Annals of Probability},
 number = {1},
 pages = {97--124},
 publisher = {Institute of Mathematical Statistics},
 title = {Second-Order Regular Variation and Rates of Convergence in Extreme-Value Theory},
 urldate = {2024-04-27},
 volume = {24},
 year = {1996}
}

@article{Holst90,
author = {L. Holst AND S. Janson},
title = {{Poisson Approximation Using the {S}tein-{C}hen Method and Coupling: Number of Exceedances of {G}aussian Random Variables}},
volume = {18},
journal = {The Annals of Probability},
number = {2},
publisher = {Institute of Mathematical Statistics},
pages = {713 -- 723},
year = {1990},
doi = {10.1214/aop/1176990854},
URL = {https://doi.org/10.1214/aop/1176990854}
}

@article{Hua11,
  title={Second order regular variation and conditional tail expectation of multiple risks},
  author={Hua, L. and Joe, H.},
  journal={Insurance: Mathematics and Economics},
  volume={49},
  number={3},
  pages={537--546},
  year={2011},
  publisher={Elsevier}
}

@Article{Kusumoto20,
  author  = {H. Kusumoto AND A. Takeuchi},
  title        = {Remark on Rates of Convergence to Extreme Value Distributions via the {S}tein Equations},
  journal      = {Extremes},
  year         = {2020},
  volume       = {23},
  number       = {3},
  pages        = {411--423},
  doi          = {10.1007/s10687-020-00380-5},
  issn         = {1386-1999},
  publisher    = {Springer}
}

@Article{Ley17,
  author       = {C. Ley and G. Reinert and Y. Swan},
  title        = {Stein's method for comparison of univariate distributions},
  number       = {none},
  volume       = {14},
  year         = {2017},
  doi          = {10.1214/16-PS278},
  journal      = {Probability Surveys},
  publisher    = {Institute of Mathematical Statistics},
}

@article{Mansanarez25,
author = {P. Mansanarez and G. Poly and Y. Swan},
year = {2025},
month = {10},
title = {Stein’s method for {F}réchet approximation: a regularly varying functions
approach},
journal = {arXiv preprint arXiv:2510.14016},
}

@article{Pickands68,
  title={Moment convergence of sample extremes},
  author={Pickands III, J.},
  journal={The Annals of Mathematical Statistics},
  pages={881--889},
  year={1968},
  publisher={JSTOR}
}

@article{Rootzen84,
title = {Attainable rates of convergence of maxima},
journal = {Statistics \& Probability Letters},
volume = {2},
number = {4},
pages = {219-221},
year = {1984},
issn = {0167-7152},
doi = {https://doi.org/10.1016/0167-7152(84)90019-1},
url = {https://www.sciencedirect.com/science/article/pii/0167715284900191},
author = {H. Rootzén},
}

@Article{Smith82,
  author  = {Smith, R. L.},
  journal = {Advances in Applied Probability},
  title   = {Uniform rates of convergence in extreme-value theory},
  year    = {1982},
  number  = {3},
  pages   = {600–622},
  volume  = {14},
  doi     = {10.2307/1426676},
}

@article{Smith87,
title = {Extreme value theory for dependent sequences via the {S}tein-{C}hen method of {P}oisson approximation},
journal = {Stochastic Processes and their Applications},
volume = {30},
number = {2},
pages = {317-327},
year = {1988},
issn = {0304-4149},
doi = {https://doi.org/10.1016/0304-4149(88)90092-0},
url = {https://www.sciencedirect.com/science/article/pii/0304414988900920},
author = {R. L. Smith},
}

@Article{Stein72,
  author  = {C. Stein},
  journal = {Proceedings of the Sixth Berkeley Symposium on Mathematical Statistics and Probability},
  title   = {A bound for the error in the normal approximation to the distribution of a sum of dependent random variables},
  year    = {1972},
  pages   = {583-602},
}

@phdthesis{CostacequePhD,
  TITLE = {{Stein's method for extreme value distributions}},
  AUTHOR = {Costac{\`e}que, B.},
  URL = {https://theses.hal.science/tel-04874813},
  NUMBER = {2024IPPAT045},
  SCHOOL = {{Institut Polytechnique de Paris}},
  YEAR = {2024},
  MONTH = Dec,
  TYPE = {Theses},
  HAL_ID = {tel-04874813},
  HAL_VERSION = {v1},
}

@phdthesis{Dobler12,
  title={New developments in {S}tein's method with applications},
  author={D{\"o}bler, C.},
  year={2012},
  school={Dissertation Bochum Ruhr-Universit{\"a}t Bochum},
}

@Book{Embrechts13,
  author    = {Embrechts, P. and Kl{\"u}ppelberg, C. and Mikosch, T.},
  publisher = {Springer Berlin Heidelberg},
  title     = {Modelling Extremal Events: for Insurance and Finance},
  year      = {2013},
  isbn      = {9783540609315},
  series    = {Stochastic Modelling and Applied Probability},
  lccn      = {97012308},
}

@book{Erdelyi56,
  author    = {A. Erdélyi},
  title     = {Asymptotic Expansions},
  publisher = {Dover Publications},
  address   = {New York},
  year      = {1956},
  pages     = {vi+108},
  series    = {Dover Books on Mathematics}
}

@phdthesis{Feidt13,
  TITLE = {{Stein's Method for Multivariate Extremes}},
  AUTHOR = {Feidt, A.},
  URL = {https://arxiv.org/abs/1310.2564},
  SCHOOL = {{Université de Zürich}},
  YEAR = {2013},
  TYPE = {Theses},
}

@book{deHaan07,
  title={Extreme Value Theory: An Introduction},
  author={de Haan, L. and Ferreira, A.},
  isbn={9780387344713},
  lccn={2006925909},
  series={Springer Series in Operations Research and Financial Engineering},
  year={2007},
  publisher={Springer New York}
}

@Book{Resnick87,
  author    = {Resnick, S. I.},
  publisher = {Springer New York},
  title     = {Extreme Values, Regular Variation and Point Processes},
  year      = {1987},
  doi       = {https://doi.org/10.1007/978-0-387-75953-1},
}
\bibliographystyle{alpha}

\end{document}